\documentclass [12pt] {article} 
\usepackage[affil-it]{authblk}

\usepackage{amsrefs} 

\usepackage{comment}
\usepackage{graphicx}
\usepackage{amssymb}
\usepackage{amsmath}
\usepackage{esint} 
\usepackage[section]{placeins}
\usepackage{mathrsfs}  
\usepackage{multirow}
\usepackage{amsthm}
\usepackage{float}
\usepackage{bbm}
\usepackage{colortbl}
\usepackage{xcolor}
\usepackage{graphicx}

\usepackage{mwe}
\usepackage{subfig} 
\usepackage{mathtools}
\usepackage[hypertexnames=false]{hyperref}
\newtheorem{thm}{Theorem}[section]
\newtheorem{prop}[thm]{Proposition}
\newtheorem{lem}[thm]{Lemma}

\newtheorem{rem}[thm]{Remark}

\newcommand{\be}{\begin{equation}}
\newcommand{\ee}{\end{equation}}

\newcommand{\sgn}{\text{sign}}

\usepackage{datetime} 
\usepackage[T1]{fontenc}
\usepackage[english] {babel}

\let\Re=\undefined\DeclareMathOperator{\Re}{Re}

\begin{document}

\baselineskip=15pt

\title{Splitting methods for Intermediate Long Wave and perturbed Benjamin--Ono models}

\author{
Yvonne Alama Bronsard%
\thanks{Nantes Université, CNRS, Laboratoire de Mathématiques Jean Leray, F-44000 Nantes, France, 
\texttt{yvonne.alamabronsard@univ-nantes.fr}},\  
Clémentine Courtès%
\thanks{Institut de Recherche Mathématique avancée, UMR 7501, Université de Strasbourg et CNRS, 7 rue René Descartes, 67000 Strasbourg, France $\&$ INRIA Nancy-Grand Est, MACARON Project, Strasbourg, France,
\texttt{courtes@math.unistra.fr}}\ ,
Benjamin Melinand%
\thanks{Institut de Recherche Mathématique avancée, UMR 7501, Université de Strasbourg et CNRS, 7 rue René Descartes, 67000 Strasbourg, France,
\texttt{melinand@unistra.fr},}
}

\maketitle

\begin{abstract}
Recently, novel integrability techniques, and most notably Birkhoff coordinates followed by an explicit formula, have been introduced to solve the Benjamin--Ono (BO) equation, leading to major theoretical and computational advances.
We propose to perturb this explicit BO formula via
a novel splitting method for numerically solving a class of PDEs, which includes the Intermediate Long Wave (ILW) equation, as well as other quasilinear equations that are not necessarily integrable. 
Using the Birkhoff coordinates of BO we establish first-order convergence in the $H^s$-norm ($s \ge 0$) while requiring only one additional derivative on the initial data $u_0 \in H^{s+1}_{0}$, significantly improving upon the regularity requirements of typical splitting methods for nonlinear dispersive equations.
Furthermore, we prove that in the deep-water limit, the scheme for the ILW equation converges to the BO solution. Computational advantages of these schemes are shown in simulations: unlike classical splitting methods, our approach does not require a restrictive quadratic time step condition to nearly preserve energy over long time scales, thus rendering efficient and accurate long-time simulations feasible. As an application, we numerically explore the soliton resolution conjecture for both the ILW and a KdV--BO equation.
\end{abstract}


\section{Introduction}

In this paper, we consider the approximation to the following class of perturbed Benjamin--Ono equation
\begin{equation}
\partial_t u = \partial_x |D| u - 2u \partial_x u + A u, \quad  u_{|t=0}(x) = u_0(x), \quad (t,x) \in \mathbb{R} \times \mathbb{T}, 
\label{eq:PDE}
\end{equation}
where $u(t,x)$ is real-valued, $D = \frac{1}{i}\frac{d}{dx}$, $\widehat{|D|f}(k) = |k|\widehat{f}(k)$ and \( A \) is a Fourier multiplier of order zero satisfying the following conditions
\begin{equation}\label{cond_a}
\widehat{Af}(k) = a(k) \widehat{f}(k), \quad a(-k) = \overline{a(k)}, \quad k \in \mathbb{Z} \quad \text{and } \|a\|_{\ell^\infty} = \sup_{k \in \mathbb{Z}} |a(k)| < \infty.
\end{equation}
We emphasize that no smallness assumptions or smoothing properties are required for the perturbative term $A$; for relevant PDE results in the former setting we refer to Remark \ref{rem:KAM}.

We will be considering initial data of nonnegative Sobolev regularity, for which global well-posedness for $u_0 \in H^s_{0} := \{ u_0 \in H^s(\mathbb{T}) : \int_{\mathbb{T}} u_0 = 0\}$, $s\ge 0$, has recently been obtained in Gassot--Laurens \cite{GL-25}.
For a discussion on the zero-mean condition, see Remark \ref{rem:zero-mean}.

We start by giving three examples of equations belonging to the class \eqref{eq:PDE}.
\smallskip

\noindent
{\bf The intermediate long wave equation.}
The (ILW) equation in the moving frame reads as
\begin{equation}\tag{ILW}\label{ILW}
\partial_t u = -i \coth(\delta D) \partial_x^2 u -2 u \partial_x u,
\end{equation}
with $\delta>0$. We can write \eqref{ILW} in the form \eqref{eq:PDE} with 
\begin{equation}\label{F:ILW}
a(k) =  i (\coth(\delta k) - \sgn(k)) k^2.
\end{equation}

\noindent
Equation \eqref{ILW} is an integrable model for long internal waves in fluids of finite total depth; we refer to Klein--Saut \cite{KS-21}*{Chapter 3} for a comprehensive overview of existing results. Regarding well-posedness, the recent work of Chapouto--Li--Oh--Pilod \cite{CLOP-24} establishes global well-posedness in $H^s(\mathbb{T})$ for $s \ge 0$ by exploiting the regularizing effect of the perturbation $A$. Subsequently, Gassot--Laurens \cite{GL-25} extended this global well-posedness theory to $s > -1/2$, which was identified as the scaling-critical regularity in Chapouto--Forlano--Li--Oh--Pilod \cite{chapouto2024intermediate}, and for which uniform-in-time a priori bounds were recently established by Harrop-Griffiths--Killip--Vi\c{s}an \cite{HKV-26}.
\smallskip

\noindent
{\bf The Smith equation.}
The equation reads
\begin{equation}
\tag{Smith}\label{Smith}
   \partial_t u =  \partial_x(1+|D|^2)^{\frac{1}{2}}u -  2 u \partial_x u.
\end{equation}
We can write \eqref{Smith} in the form \eqref{eq:PDE} with
\begin{equation*}
a(k)    = i k \sqrt{1+k^2}-i\, \sgn(k)k^2.
\end{equation*}

\noindent
This equation is a model for continental-shelf waves \cite{smith-72}, which is not known to be integrable, see also \cite{KS-21} and references therein. 
\newline

The last example we present is a recent model which can be seen as a Korteweg de Vries (KdV) perturbation of the \eqref{BO} equation and can be derived following a similar approach to that in \cite{M-26}. It is also not known to be integrable.
\smallskip

\noindent
{\bf A KdV--BO equation.} 
The equation is given by
\begin{equation}\tag{KdVBO}\label{KdVBO}
\partial_t u =- \frac{\partial_x + \partial_x^3}{\sqrt{1+|D|^2}} u - 2 u \partial_x u,
\end{equation}
which can we written in the form \eqref{eq:PDE} with 
\begin{equation*}
a(k)    = -i \frac{k - k^3}{\sqrt{1+k^2}} -i\, \sgn(k)k^2.
\end{equation*}

We notice that all the above examples of perturbations are in the case where $A$ is skew-symmetric. 
\newline

In the case where $a(k) = 0$ for all $k\in\mathbb{Z}$, we recover the Benjamin--Ono equation. Both the design and the convergence analysis for our scheme are based on new tools for solving this equation.
\smallskip

\noindent
{\bf The Benjamin--Ono equation.} The \eqref{BO} equation reads
\begin{equation}\tag{BO}\label{BO}
\partial_t v = \partial_x |D|  v - 2v \partial_x v.
\end{equation}
This equation models the propagation of unidirectional long internal waves at the interface of a two-layer fluid where the lower layer is of infinite depth, as rigorously justified by Paulsen~\cite{Paulsen2024}.

Using techniques from the theory of integrable systems, numerous advances have been obtained for \eqref{BO}. First, through the discovery of Birkhoff coordinates \cite{GK21}, a nonlinear Fourier transform, Gérard--Kappeler--Topalov \cite{GKT-23} demonstrated the global well-posedness of \eqref{BO} in $H^s(\mathbb{T})$ for $s>-\frac{1}{2}$ and ill-posedness otherwise. This nonlinear transformation will be essential for the convergence analysis of our scheme, see Section \ref{sec:strategy-pf}. Furthermore, circumventing the need for any nonlinear transform, Gérard \cite{G-23} established the existence of an explicit formula for solving \eqref{BO}. This remarkable formula allows one to express the solution at time $t$ in terms of the initial data $v_0$, through the composition of linear operators. Here, we present it as a characterization of the Fourier coefficients of the solution (see also Remark \ref{rem:formula-holomorphic}):
\begin{equation}\label{eq:explicitForm}
\widehat{v}(t,k) =\left\langle (e^{it}e^{2itL_{v_0}}S^{*})^k \Pi v_0, 1 \right\rangle, \quad k \ge 0,
\end{equation}
with  the Lax and shift operators
\[
{L}_{v_0}f = Df - \Pi(v_0f), \quad  S^{*}f= \Pi (e^{-ix}f), \quad f\in L^2_+,
\]
where $\widehat{\Pi f}(k) = \mathbbm{1}_{k\ge 0}\widehat{f}(k)$ and 
$L^2_{+} = \Pi L^2 = \bigl\{f\in L^2: \widehat{f}(k) = 0, \ k<0 \bigr\}$. 
Moreover, since $v$ is real-valued, we simply have $\widehat{v}(t,k) = \overline{\widehat{v}(t,-k)}$ for $k<0$. 

Recently, the works \cite{ABCD-25} and \cite{ABL-26} demonstrated the significant advantage of approximating this formula \eqref{eq:explicitForm} rather than applying existing numerical methods to \eqref{BO}. Namely, these works show that approximating \eqref{eq:explicitForm} yields schemes that are exact in time (i.e. they do not require a time discretization) as the formula provides an explicit representation of the solution at any time $t \in [0,T]$. Thus, not only do these schemes solely require a spatial discretization to be computable, but their computational cost is independent of the final time $T$
(see equation \eqref{eq:ABCD-scheme} and Remark \ref{rem:ABCD-err} for more detail). This observation is the starting point for the {\it construction} of our scheme: as \eqref{BO} can be computed exactly in time, yielding stable and spectrally accurate schemes, we propose to iteratively solve the \eqref{BO} flow and compose it with the flow generated by the perturbation over small time intervals, leading to a highly efficient and stable computational approach for solving \eqref{eq:PDE}. This is further detailed in the next section.

\subsection{Splitting schemes}\label{sec:splittingscheme}
We consider approximating \eqref{eq:PDE} using a novel first-order splitting scheme, which partitions the terms in a manner distinct from traditional Lie splitting methods. Specifically, we decompose the right-hand side of \eqref{eq:PDE} into the Benjamin--Ono part \( B(u) = \partial_x |D| u - 2u \partial_x u\) and the linear perturbation \( A u \):
\begin{align*}
\partial_t u = B(u) + Au, \ u(0) = u_0; \quad
\partial_t v = B(v), \ v(0) = v_0;
\quad
\partial_t z = Az,\  z(0) = z_0,
\end{align*}
where 
\[
u(t) = \Psi^{t}_{\rm full}(u_0), \quad  v(t) = \varphi^{t}_{\rm BO}(v_0) \quad \text{and} \quad  z(t) = e^{tA}z_0 \]
denote the solutions at time \( t \) to the respective initial value problems defined in $H^s_0$, $s\ge 0$.

A crucial advantage of this decomposition (as opposed to classical Lie splittings) is that both sub-flows $\varphi^{t}_{\rm BO}$ and $e^{tA}$ can be computed {\it exactly in time} and are well defined in $H^s_0$ {\it for all} $t\ge 0$. Indeed, Gérard's explicit formula \eqref{eq:explicitForm} provides a closed-form expression for \( \varphi_{\rm BO}^t w_0 \), and \cites{ABCD-25, ABL-26} allow for its efficient implementation in Fourier space. For the linear part, since \( A \) is a diagonal operator in Fourier space, the evolution \( e^{tA}z_0 \) is evaluated exactly in this basis. Consequently, the only time-discretization error in our method stems from the splitting procedure itself. Our scheme is presented next.

\medskip
\noindent
\textbf{The perturbative splitting scheme.}
Let \( \tau \in (0, 1] \) be the time step. Our splitting scheme approximates the solution at time \( t_{n+1} = t_n + \tau \) by composing these exact flows over each time step:
\begin{equation}
u^{n+1} =  \varphi^{\tau}_{\rm BO} e^{\tau A} u^n, \quad u^0 = u_0.
\label{eq:Lie}
\end{equation}

For a discussion on the implementation of the fully-discrete scheme we refer to Section \ref{sec:numerics} (see also Remark \ref{rem:space-approx}).

\begin{rem} We emphasize that for initial data \( u_0 \in H^s_0 \) satisfying the zero-mean condition, i.e., \( \int_{\mathbb{T}} u_0 = 0 \), the splitting schemes (as well as its fully discrete counterpart) exactly preserve this property, ensuring that \( \int_{\mathbb{T}} u^n = 0 \) for all \( n \ge 0 \). This is necessary to obtain convergence, which we establish in $H^s_0$, by exploiting the analyticity properties of the Birkhoff map in these spaces \cite{GKT-21, GL-25}, see Section \ref{sec:strategy-pf} for more detail.
\end{rem}

Next, we motivate our choice of splitting by comparing it with classical Lie splitting schemes.

\medskip
\noindent
{\bf Comparison with the classical Lie splitting method.} 
Two effects are present in \eqref{eq:PDE}: the Burgers nonlinearity $\partial_x(u^2)$, which set alone forms shocks in finite time, and the linear terms $\mathcal{L} := \partial_x|D|+ A$, whose evolution is bounded in all Sobolev norms on any compact time interval, and even preserves the $L^2$ norm when $A$ is skew-adjoint, as is the case in all the above examples. Consequently, in the error analysis of time-discretization schemes for \eqref{eq:PDE}, the central challenge lies in establishing stability bounds to control the derivative loss in the nonlinearity.

The Lie splitting scheme relies on decomposing the governing equation $\partial_tu - \mathcal{L}u =- \partial_x (u^2)$ into the linear equation $\partial_tu = \mathcal{L}u$ and the inviscid Burgers equation $\partial_tu = -\partial_x (u^2)$. The scheme advances the solution over a time step $\tau$ by composing their respective flows, yielding 
\begin{equation}\label{eq:lie-splitting}
u^{n+1} = (\Phi_{\mathcal{N}}^\tau \circ \Phi_{\mathcal{L}}^\tau)(u^n),
\end{equation}
where $\Phi_{\mathcal{L}}^\tau = e^{\tau \mathcal{L}}$ is the (exact) linear propagator and $\Phi_{\mathcal{N}}^\tau$ is the nonlinear Burgers flow.

The error analysis of such splitting methods is typically conducted assuming that the Burgers flow $\Phi_{\mathcal{N}}^\tau$ is computed exactly in time, and the key difficulty in the analysis lies in the fact that Burger's flow eventually produces discontinuities independently of the smoothness of the initial data, as the underlying
equation is not well posed in any Sobolev space with positive regularity. 
For smooth solutions we further refer to \cites{HKRT-11, HLR-13} in the context of the KdV equation and to \cite{DHKR-15} in the context of the \eqref{BO} equation, where energy estimates are employed. For the approximation to low-regularity solutions to KdV we further refer to \cite{rousset-22}, which introduces discrete Bourgain spaces and a quadratic time step condition to guarantee stability of the method.

We emphasize that, a priori, our analysis is even more challenging than that for the KdV equation. While the KdV equation features the third-order Airy dispersive operator $\partial_{xxx}$, our linear dispersive operator $\partial_x|D| + A$ is only of second order. Consequently, the high-frequency dispersion present in the equation is significantly reduced, rendering the control of the derivative in the nonlinearity a much harder problem. Unlike the KdV equation, where the dispersion is strong enough to permit semilinear arguments, when working with solutions that lack high smoothness (for which energy estimates cannot be applied), this weaker dispersion cannot overcome the derivative loss induced by the nonlinearity using standard Duhamel or Picard arguments. As a result, the equations studied here are of quasilinear nature, rendering the analysis more challenging and requiring different techniques. 

To enable the analysis of our splitting scheme \eqref{eq:Lie}, we introduce a novel approach that uplifts the Birkhoff coordinates, see Sections \ref{sec:strategy-pf} and \ref{sec:convergenceresults}. Remarkably, this, together with our choice of splitting, allows us to obtain a significantly better convergence result than that of classical splitting methods, even when compared to the KdV, BO, and NLS equations.

\subsection{Main results}

First, we establish the convergence of the splitting scheme \eqref{eq:lie-splitting} to the exact solution of \eqref{eq:PDE} as the time step $\tau$ goes to zero: by requiring one additional Sobolev regularity on the solution, we obtain first-order convergence in time.

\begin{thm}[Convergence of the perturbative splitting scheme]\label{thm:cv-time}
Let $s\ge 0$, $T>0$, $u_0\in H^{s+1}_0$, and consider $A$ satisfying \eqref{cond_a}. Let $u \in \mathcal{C}([0,T], H^{s+1}_0)$ be the solution of \eqref{eq:PDE} with initial datum $u_0$, and let $u^n$ denote the splitting scheme \eqref{eq:Lie}. Then there exist a threshold step-size $\tau_0 \in (0,1]$ and a constant $C>0$ such that for any $\tau \in (0,\tau_0]$,
\[
\sup_{n\tau \le T}\|u(n\tau)-u^n \|_{H^s} \le C \| a \|_{\ell^{\infty}} \tau.
\]
\end{thm}
\noindent The threshold $\tau_0$ is given explicitly in \eqref{eq:tau_0}, and both $\tau_0$ and $C$ depend continuously on $s$, $\| u_0\|_{H^{s+1}}$, $\|a \|_{\ell^{\infty}}$, and $T$. Furthermore, $C$ grows exponentially with $T$, see Remark \ref{rem:exp-dep}.

\medskip

Second, in the case of the \eqref{ILW} equation, we provide an additional result regarding the deep-water limit from \eqref{ILW} to \eqref{BO} as $\delta \to \infty$. This limit was rigorously established on the continuous level, see for example the recent works \cites{CLOP-24, GL-25}. We now present a discrete analogue for our semi-discrete scheme: we obtain convergence of the splitting scheme for the ILW equation to the BO equation $\varphi^{t}_{\rm BO}$, with explicit rates in terms of the depth parameter $\delta$ and the time step $\tau$.

\begin{thm}[Convergence of the ILW scheme to BO]\label{thm:cv-ILW-BO}
Let $s\ge 0$, $T>0$ and $u_0 \in H^{s+1}_{0}$. Let $u \in \mathcal{C}([0,T], H^{s+1}_0)$ be the solution of \eqref{ILW} with initial datum $u_0$ and perturbation $A$ defined by \eqref{F:ILW} with parameter $\delta$. Let $u^n$ denote the splitting scheme \eqref{eq:Lie}. Then there exist a threshold $\delta_0 \ge 1$ and constants $C>0$ and $C_0>0$ such that for all $\delta \ge \delta_0$ and any $\tau \in (0,1]$, we have
\[
\sup_{n\tau \le T}\|u^n-\varphi^{n\tau}_{\rm BO} u_0 \|_{H^s} \le \left(C_0 T e^{C_0 T} + C  \tau\right) e^{-2 \delta}.
\]
\end{thm}
\noindent
Both $\delta_0$ and $C$ depend continuously on $s$, $\|u_0\|_{H^{s+1}}$, and $T$, while $C_0$ depends continuously on $s$, $\|u_0\|_{H^{s}}$, and $T$. For a discussion on the threshold requirements on $\tau$, we refer to Remark~\ref{rem:tau_0}.
\smallskip

Next, we compare our convergence results to those of classical splitting schemes for the KdV and BO equations, which also have Burgers nonlinearity, and then to existing numerical schemes for the ILW equation.
\smallskip

\noindent
{\bf Error analysis of classical splitting methods for KdV and BO.} For an analysis at second order of the Strang splitting method for solving the KdV equation, \cite{HKRT-11} and \cite{HLR-13} show convergence at second order in $H^s$-norm  for $u_0 \in H^{\alpha}$, $\alpha = s + 5$, $s \geq 1$, under the assumption that Burgers equation is solved exactly. Applying the same analysis on the Lie splitting method would yield first-order convergence for $\alpha = s + 3$ \textit{i.e.} requiring three additional derivatives on the solution. For convergence at a fractional order $0<\nu<1$ less than one, we refer to the work of Rousset--Schratz \cite{rousset-22} which employ discrete Bourgain spaces and Duhamel's formula to obtain the regularity condition $\alpha = s + 3\nu$ of the Lie splitting method, extending the prior results to this fractional low-regularity setting. We note that their stability analysis requires a quadratic time step condition.

Furthermore, for the \eqref{BO} equation, Dutta--Holden--Koley--Risebro \cite{DHKR-15} show $L^2$ convergence of the Lie splitting method for $u_0 \in H^{5/2}$, requiring $2.5$ orders of extra regularity: two additional derivatives come from the linear operator of order two, while the remaining half-derivative arises from their stability bound, which is based on Tao's gauge transform (see \cite[Theorem 2.2]{DHKR-15} for details). We note that this work preceded the discovery of the explicit formula \eqref{eq:explicitForm}; the recent numerical results \cites{ABCD-25, ABL-26} now demonstrate that no time discretization is required to solve this equation.

In general, achieving first-order convergence in $H^s$ for a Lie splitting method requires, at best, the regularity $\alpha = s + \theta$, where $\theta$ denotes the order of the linear operator. In contrast, our new choice of splitting yields $\alpha = s +1$. 
\smallskip

\noindent
{\bf Error analysis results for ILW.}
Regarding prior error analysis results for the ILW equation, existing literature has used discrete energy methods for analyzing pseudo-spectral discretizations in space coupled with finite-difference approximations in time. This includes the $H^{1/2}$ energy norm estimates by Deng--Ma \cite{DM09}, and the $L^2$-error bounds in Pelloni--Dougalis \cite{PD01}, which necessitate highly smooth initial data ($u_0 \in H^r$ for $r \ge 7$). Here, we depart from these energy techniques. Instead, we exploit the underlying structure of this PDE class by treating the equation as a perturbation of the Benjamin-Ono (BO) equation, for which significant insight into its integrability structure has recently been obtained. This approach yields significantly sharper results, allowing for an error analysis in any $H^s$-norm ($s \ge 0$) for initial data in $H^{s+1}_0$.

\begin{rem}[On the behavior of small perturbations of integrable systems]\label{rem:KAM}
Small perturbations of integrable equations have been studied via extensions of KAM theory. For studies concerning perturbations of the BO equation on the torus, we refer to \cites{LY-11,MZ-14}, as well as \cite{BG-21}, which considers generalized Benjamin--Ono equations with generic small initial data, and \cite{BG-24}, which addresses small Hamiltonian perturbations of the BO equation using the Birkhoff coordinates of BO (see Section \ref{sec:birkhoffcoord}). These coordinates will play a central role in the analysis of the splitting schemes presented in Section \ref{sec:splittingscheme} for solving the class of equations \eqref{eq:PDE}.
\end{rem}

\begin{rem}[Zero-mean condition]\label{rem:zero-mean}
Recall that $H^s_0$ denotes the subspace of $H^s(\mathbb{T})$ consisting of functions with zero mean.
As the mean value $c = \fint_{\mathbb{T}} u dx$ is a conserved quantity for \eqref{BO}, we have by Galilean symmetry that $\tilde u (t,x) = u(t, x+2 ct) - c$ is also solution of \eqref{BO}. Hence,
we can enforce the
condition that the solution to \eqref{BO} belongs to $H^s_0$. Similarly, we can impose this condition to solutions of \eqref{eq:PDE}, by considering
\[
\tilde u (t,x) = u\left(t, x+2 c\frac{e^{a(0)t}-1}{a(0)}\right) - ce^{a(0)t},
\]
see \cite{GL-25}*{Section 2.2}. As discussed previously, our schemes also satisfy the zero-mean condition.
\end{rem}

\begin{rem}[Explicit formula on the disk]\label{rem:formula-holomorphic}
The explicit formula on the torus \cite{G-23}*{Theorem 4} is given by
\begin{equation}
\label{eq:explicit-BO-disc}
\Pi v(t, z)=\left\langle\left(I-z e^{i t} e^{2 i t L_{v_0}} S^*\right)^{-1} \Pi v_0, 1\right\rangle, \quad \forall z \in \mathbb{D} = \{z\in \mathbb{C} : |z| <1 \}.
\end{equation}
Nevertheless, as uncovered by \cite{ABCD-25}, for numerical computations, the Fourier variable formulation \eqref{eq:explicitForm} is the most suitable.
\end{rem}

\begin{rem}[Growth of the error constant in time]\label{rem:exp-dep}
While the \eqref{BO} scheme introduced in \cite{ABCD-25} and generalized in \cite{ABL-26} achieves an optimal linear error growth with respect to the final time $T$ due to the exact-in-time evaluation of the explicit formula (see Remark \ref{rem:ABCD-err}), our scheme \eqref{eq:Lie} for the perturbed PDE \eqref{eq:PDE} naturally introduces an exponential dependence on~$T$. This stems from the necessity of introducing a time stepping scheme for solving the equation; as is standard for such approximations, establishing convergence requires combining local error bounds with a stability estimate via a Gronwall argument, which yields the exponential growth factor in $T$.

Nevertheless, we note that on the unbounded domain $\mathbb{R}^d$, uniform-in-time error bounds can be achieved for splitting approximations. This was first demonstrated by Carles--Su \cite{CS22b} for the NLS equation, where the authors elegantly leveraged scattering theory and dispersive decay estimates---analytical tools that fundamentally rely on the properties of an unbounded geometry.
\end{rem}

\subsection{Organization of the paper}

The remainder of this paper is organized as follows. In Section~\ref{sec:strategy-pf}, we outline the overall strategy of the proof, detailing the primary technical challenges and motivating the need for uplifting Birkhoff coordinates. In Section~\ref{sec:birkhoffcoord}, we review the fundamental properties of Birkhoff coordinates.
Sections~\ref{sec:stab-bounds} and \ref{sec:local-error} establish the crucial stability estimates in Birkhoff space and the local error bounds, respectively. These results are combined in Section~\ref{sec:proof_thm_1} to prove the convergence of the splitting schemes (Theorem~\ref{thm:cv-time}). In Section~\ref{sec:comparison-ILW-BO}, we establish the convergence of the scheme for \eqref{ILW} in the deep-water limit (Theorem~\ref{thm:cv-ILW-BO}). Finally, Section~\ref{sec:numerics} presents numerical experiments illustrating the performance and long-time behavior of the scheme.

\section{Strategy of the proof}\label{sec:strategy-pf}

The first-order convergence analysis of time discretizations for nonlinear wave-type equations relies on two key ingredients: a second-order {\it local error}\footnote{Bounding the local error terms requires higher regularity on the solution than the underlying error norm (here $H^s$). This regularity requirement is determined by the choice of the scheme.} 
estimate $\mathcal{O}(\tau^2)$; and a {\it stability bound} for the numerical flow, established in the same error norm. Combining these two bounds via a  Gronwall inequality yields convergence on bounded time intervals $[0, T]$, under the regularity assumptions required by the local error analysis.

For our class of splitting schemes, carrying out this program directly in classical Sobolev spaces or Bourgain spaces fails due to the quasilinear nature of our equations \eqref{eq:PDE}. Instead, we map the discrete dynamics into Birkhoff coordinates via the nonlinear Fourier transform $\Phi: H^s_0 \to \mathfrak{h}^{s+1/2}$ (see Section~\ref{sec:birkhoffcoord} and \cite{GL-25, GK21,GKT-21}). In these coordinates, we obtain the necessary stability bound as detailed next.
\medskip 

A central technical requirement of our proof lies in establishing that over a time step $\tau \in (0,1]$, the flow of the splitting scheme $g \mapsto \varphi^\tau_{\rm BO}e^{\tau A}g$ is locally Lipschitz with a Lipschitz constant of the precise form $1 + c\tau \le e^{c\tau}$. Establishing both the leading constant of exactly~$1$ (rather than $C>1$) and the linear dependence $c\tau$ in time is only possible in Birkhoff coordinates, which is the primary reason why working in these coordinates is indispensable for our analysis. 
Indeed, we extract the linear $\mathcal{O}(\tau)$ dependence in the stability bound for the splitting scheme in Lemma~\ref{lem:stab-split} by integrating the perturbed dynamics in Birkhoff coordinates
and combining it with Lemma~\ref{lem:stab-full}, which proves that for the Benjamin--Ono flow itself, the mapping $v \mapsto \varphi^t_{\rm BO} v$ is locally Lipschitz in $\mathfrak{h}^{s+1/2}$ with constant $1 + \mathcal{O}(t)$.

We emphasize that 
since $n=T/\tau$ time steps are iterated, any prefactor $C>1$ applied at each step would compound exponentially to $C^{T/\tau} \to \infty$ as $\tau \to 0$. Because the Birkhoff map $\Phi$ is nonlinear and not an isometry, mapping this single-step bound back into $H^s$ introduces a (non-unit) constant. Consequently, the entire global stability iteration and induction must be performed on the Birkhoff side.
\medskip

Finally, the local error analysis in Section~\ref{sec:local-error} (Lemma~\ref{eq:loc-err}) connects the exact and numerical flows by writing different Duhamel representations for them, and then comparing them using Lemma~\ref{lem:comp_BO_Eq}, which bounds the difference between the full flow $\Psi^t_{\rm full}$ with $\varphi^t_{\rm BO}$ on the Birkhoff side. 
To close the convergence argument in Section~\ref{sec:proof_thm_1}, we employ an inductive argument in $\mathfrak{h}^{s+1/2}$. This induction establishes that the numerical iterates $w_k = (\varphi_{\rm BO}^\tau e^{\tau A})^ku_0$ remain confined to a bounded ball in $\mathfrak{h}^{s+1/2}$,
ensuring that the stability estimates hold uniformly across all steps $k=1,\dots,n$, without losing constants.
We note that these proof techniques contrast with prior frameworks for proving convergence, which up until now have been focused on solving semilinear PDEs, and we refer to Remark \ref{rem:bourgain-contrast}.
\newline

Lastly, Theorem~\ref{thm:cv-ILW-BO} establishes convergence in the deep-water limit ($\delta \to \infty$) for the ILW scheme by combining our main time-discretization result (Theorem~\ref{thm:cv-time}) with Proposition~\ref{prop:comparison_ILW_BO}, which bounds the difference between the continuous ILW and BO flows on the Birkhoff side with explicit exponential rates $e^{-2\delta}$.

\begin{rem}\label{rem:bourgain-contrast}

For classical time stepping methods (such as splitting schemes or exponential integrators) applied to semilinear dispersive PDEs, the pre-factor $1+O(\tau)$ is obtained using the scheme: the time step factor $\tau$ is extracted directly from the first-order term in front of the non-linear operator in the scheme, where the linear part (i.e., the zeroth-order term) of the flow is unitary. For our perturbed splitting scheme, attempting a Taylor expansion to isolate $\tau$ yields:
\[
u^{n+1} = \varphi^\tau_{\rm BO} u^n + \tau A \varphi_1(\tau A) \varphi^\tau_{\rm BO} u^n, \quad \text{with} \quad \varphi_1(z) = \frac{e^z -1}{z}.
\]
While the second term above is of order $\mathcal{O}(\tau)$, since the flow $\varphi^\tau_{\rm BO}$ is nonlinear and non-unitary in $H^s$, one cannot bound solution differences of the leading term by $1$. Hence, due to the structure of our scheme, we cannot extract a $1+\mathcal{O}(\tau)$ factor directly from the scheme when performing the stability analysis, and must instead extract this factor analytically through the stability bound in Birkhoff coordinates.
\end{rem}

\section{Convergence results}\label{sec:convergenceresults}
\subsection{Birkhoff coordinates and Duhamel's formula}\label{sec:birkhoffcoord}

{\bf Birkhoff coordinates.}
Key to our analysis are the Birkhoff coordinates for the \eqref{BO} equation, led by the mapping $\Phi$ which we define as follows
\begin{equation*}
\Phi:v\in H^s_0\mapsto (\zeta_n(v))_{n\geq 1}\in \mathfrak{h}^{s+\frac 12},
\end{equation*}
with
\begin{equation*}
\mathfrak{h}^{ s+\frac 12}
    = \bigg\{ (z_n)_{n\geq 1}\in \mathbb{C}^\mathbb{N} : \|z_n\|_{\mathfrak{h}^{s+\frac12}}^2 = \sum_{n\geq 1} n^{2s+1}|z_n|^2<+\infty \bigg\}.
\end{equation*}
The advantage of mapping the solution $v$ of the \eqref{BO} equation into these Birkhoff coordinates was demonstrated by \cites{GK21}, where the authors show that for any $n\ge 1$, the Birkhoff coordinates of  the solution $v$ of \eqref{BO} satisfy the differential equation
\begin{equation}
\frac{d}{d t} \zeta_n(v(t)) = i\omega_n(v(t))\zeta_n(v(t)) ,
\label{eq:BO-Birk-eqn}
\end{equation}
where
\begin{align}\label{eq:omega-n}
\omega_n(v) = n^2 - 2\sum_{k\ge 0} \text{min}\{ k,n\}| \zeta_k(v)|^2.
\end{align}
Hence, by letting $\Omega_n(t;v) = \int_0^t \omega_n(v(s))ds$ we have
\[
\zeta_n(v(t)) = e^{i\Omega_n(t;v)}\zeta_n(v_0).
\]

Moreover, the mapping $\Phi$ enjoys several useful properties. In particular, it is continuously differentiable, and we denote its (Fréchet) derivative at $v$ in the direction $f$ by
\[
d_v \Phi[f] \;=\; \left.\frac{d}{ds}\,\Phi(v + s f)\right|_{s=0}.
\]
With this notation in place, we proceed to state the main properties of the Birkhoff map, which can be found in \cite{GKT-21}.

\begin{prop}\label{t:birkhoff-bdd}

Let $s\ge 0$. The Birkhoff map sends bounded subsets of $H^s_0$ to bounded subsets of $\mathfrak{h}^{s+\frac12}$, and similarly for the inverse map.

Moreover, for every $v \in H^s_0$, there is $C>0$ and a neighborhood $V$ of $v$ in $H^s_0$ such that for every $\tilde{v}\in V$ and $f\in H^s$,
\begin{align}
\| d_v \Phi [f] \|_{\mathfrak{h}^{s+\frac12}} &\leq C \| f \|_{H^s} ,
\label{d-1} \\
\| d_v\Phi[f] - d_{\tilde{v}}\Phi[f]  \|_{\mathfrak{h}^{s+\frac12}} &\leq C\| v - \tilde{v} \|_{H^s} \| f \|_{H^s}.
\label{d-2}
\end{align}
Similar estimates hold for the inverse map $\Phi^{-1}$. 
\end{prop}

\noindent Thanks to \eqref{d-1}, for any ball $B$ of $H^s_0$ there exist two positive constants $c, C$ such that for any $f \in B$ we have
\[
c\|\Phi (f) \|_{\mathfrak{h}^{s+\frac12}}\leq \|f\|_{H^s}\leq C \|\Phi (f) \|_{\mathfrak{h}^{s+\frac12}}.
\]

Next, after mapping into Birkhoff coordinates, we present Duhamel’s formula for \eqref{eq:PDE}, as introduced in \cite{GL-25}*{Section 2.3}.
\newline

\noindent
{\bf Duhamel's formula for \eqref{eq:PDE}.} Given $s\ge 0$, take $u$ to be the solution to \eqref{eq:PDE} in $H^s_0$. By using the formula \eqref{eq:BO-Birk-eqn} for solving the \eqref{BO} equation together with an application of the chain rule we have
\begin{equation*}
\frac{d}{d t} \zeta_n(u(t)) = d_{u(t)}\zeta_n\,[\partial_t u(t)] =  i\omega_n(u(t)) \zeta_n(u(t)) + d_{u(t)}\zeta_n\,[ Au(t)] .
\end{equation*}
and hence
\begin{equation*}
\frac{d}{d t} ( e^{-i\Omega_n}  \zeta_n ) = e^{-i\Omega_n} d_{u(t)}\zeta_n [Au].
\end{equation*}
Thus, when mapped in Birkhoff coordinates, $u$ satisfies the following Duhamel formula 
\begin{equation}
\zeta_n(u(t)) = e^{i\Omega_n(t;u)}\zeta_n(u_0) + \int_0^t e^{i[\Omega_n(t;u) - \Omega_n(r;u)]}d_{u(r)} \zeta_n[A u(r)]dr.
\label{eq:duh}
\end{equation}

\subsection{Stability bounds}\label{sec:stab-bounds}
\begin{lem}[Stability of the Benjamin--Ono flow]\label{lem:stab-full} Let $M>0$ and $s\ge 0$. Let $f$ and $g$ in  $H^s_{0}$ be such that $\| f \|_{H^s},\| g \|_{H^s} \leq M$. There exists $c_1 >0$ that continuously depends on $M$ such that
\[
\| \Phi(\varphi^{t}_{\rm BO} f) - \Phi(\varphi^{t}_{\rm BO} g)\|_{\mathfrak{h}^{s+\frac{1}{2}}} \le e^{c_1 t}\|\Phi(f) - \Phi(g) \|_{\mathfrak{h}^{s+\frac{1}{2}}},
\]
for all $t\ge 0$.
\end{lem}
\begin{proof}[Proof of Lemma \ref{lem:stab-full}]%
Let $v(t) = \varphi^{t}_{\rm BO} f$ and $\tilde v(t) = \varphi^{t}_{\rm BO} g$ and $n \in \mathbb{N}$. We have
\[
\zeta_n(v(t)) = e^{i\Omega_n(t; v)}\zeta_n(f), \quad \zeta_n(\tilde v(t)) = e^{i\Omega_n(t; \tilde v)} \zeta_n(g)
\]
and using \eqref{eq:omega-n}
\begin{align*}
|e^{i\Omega_n(t;v)}-e^{i\Omega_n(t;\tilde v)}| &\le \int_0^t |\omega_n(r;v) - \omega_n(r;\tilde v)|dr\\
& \le \int_0^t \left| \sum_{k\ge 1}2k \left| |\zeta_k(v(r))|^2 - |\zeta_k(\tilde v(r))|^2 \right| \right| dr\\
&\le C\int_0^t \| \Phi(v(r)) - \Phi(\tilde v(r))\|_{\mathfrak{h}^{\frac12}}dr
\end{align*}
Hence,
\begin{equation}\label{eq:diff-zeta-bd}
\begin{aligned}
\|\Phi(v(t)) -  \Phi(\tilde  v(t))\|_{\mathfrak{h}^{s+\frac12}} & \le \sup_{n\ge 1}|e^{i\Omega_n(t; v)} - e^{i\Omega_n(t;\tilde  v)} | \| \zeta_n(f)\|_{\mathfrak{h}^{s+\frac12}} + \|\zeta_n(f) - \zeta_n(g)\|_{\mathfrak{h}^{s+\frac12}}\\
& \le c_1 \int_0^t \| \Phi(v(r)) - \Phi(\tilde v(r))\|_{\mathfrak{h}^{s+\frac12}}dr + \|\Phi(f) - \Phi(g) \|_{ \mathfrak{h}^{s+\frac12}},
\end{aligned}
\end{equation}
where  the constant $c_1$ depends  only on $\| f\|_{H^s}$ and $\| g\|_{H^s}$. We obtain the desired result by applying a Gronwall argument to the above.
\end{proof}
\begin{lem}[Stability of the splitting]\label{lem:stab-split} Let $M>0$, $s\ge 0$ and let $A$ satisfy \eqref{cond_a}. Let  $f$ and $g$ in  $H^s_{0}$ be such that $\| f \|_{H^s},\| g \|_{H^s} \leq M$. Then there exists $c>0$ that continuously depends on $M$ and $\| a \|_{\ell^{\infty}}$ such that
\[
\|\Phi\left(\varphi^{\tau}_{\rm BO}e^{\tau A}f\right) - \Phi\left(\varphi^{\tau}_{\rm BO}e^{\tau A}g\right)\|_{\mathfrak{h}^{s+ \frac12}} \le e^{c \tau}\| \Phi(f)-\Phi(g)\|_{\mathfrak{h}^{s+\frac12}},
\]
for any $\tau \in (0,1]$.
\end{lem}
\begin{proof}[Proof of Lemma \ref{lem:stab-split}]
Let $w(t) = e^{tA}f$ and $z(t) = e^{tA}g$. We have that $\partial_t w = A w$ and hence $w$ satisfies
\[
\frac{d}{dt}\zeta_n(w(t)) = \text{d}_{w(t)} \zeta_n [A w] \quad \text{and} \quad \zeta_n(w(t)) = \zeta_n(f) + \int_0^t \text{d}_{w(r)}\zeta_n[Aw(r)]dr.
\]
Similarly, $\zeta_n(z(t)) = \zeta_n(g) + \int_0^t \text{d}_{z(r)}\zeta_n[Az(r)]dr$. Thus by taking $t=\tau \le 1$ we have
\[
\| \Phi(w(\tau)) - \Phi(z(\tau))\|_{{\mathfrak{h}^{s+\frac12}}} \le \| \Phi(f) - \Phi(g)\|_{{\mathfrak{h}^{s+\frac12}}} + \int_0^{\tau} dr \|\text{d}_{w(r)} \Phi [A w(r)] - \text{d}_{z(r)} \Phi[Az(r)] \|_{{\mathfrak{h}^{s+\frac12}}},
\]
where by using \eqref{d-1} and \eqref{d-2} we have
\begin{align*}
\|\text{d}_{w(r)} \Phi [A w(r)] - \text{d}_{z(r)} \Phi[Az(r)] \|_{\mathfrak{h}^{s+\frac{1}{2}} }
&\le \| \text{d}_{w(r)}\Phi[Aw(r)] - \text{d}_{z(r)} \Phi[Aw(r)]\|_{\mathfrak{h}^{s+\frac{1}{2}}} \\
&\qquad +\| \text{d}_{z(r)} \Phi[Aw(r)] - \text{d}_{z(r)} \Phi[Az(r)] \|_{\mathfrak{h}^{s+\frac{1}{2}}}\\
&\le C(\|w(r) - z(r)\|_{H^s}\| Aw(r)\|_{H^s}  + \| A(w(r) - z(r))\|_{H^s})\\
&\le C (1+\|a\|_{\ell^{\infty}}) \|\Phi(w(r)) - \Phi(z(r))\|_{\mathfrak{h}^{s+ \frac{1}{2}}}
\end{align*}
with $C>0$ depending on $M$. 
Hence, by applying a Gronwall argument it follows that 
\[
\| \Phi(w(\tau)) - \Phi(z(\tau))\|_{\mathfrak{h}^{s+1/2}} \le e^{C(1+\|a\|_{\ell^{\infty}})\tau}\| \Phi(f) - \Phi(g)\|_{\mathfrak{h}^{s+1/2}}.
\]
Finally, by applying Lemma \ref{lem:stab-full} there exists $c_1>0$ that only depends on $Me^{\sup \Re(a)}$ such that
\[
\|\Phi(\varphi^{\tau}_{BO} w(\tau)) -  \Phi(\varphi^{\tau}_{BO} z(\tau))\|_{\mathfrak{h}^{s+\frac12}} \le e^{(c_1+C(1+\|a\|_{\ell^{\infty}}))\tau} \| \Phi(f) - \Phi(g)\|_{\mathfrak{h}^{s+\frac12}},
\]
which concludes the proof.
\end{proof}

\begin{rem}[Condition on $\tau$]
We need to bound terms of the form $e^{\tau \| a\|_{\ell^{\infty}}}$. Thus we merely needed $\tau$ to be bounded above by a constant. For simplicity we choose $\tau\in (0,1]$.
\end{rem}

\subsection{Local error analysis}\label{sec:local-error}

\begin{lem}[Local error bound]\label{eq:loc-err}
Let $s\ge 0$, $T>0$, $u_0\in H^{s+1}_0$ and $A$ satisfy \eqref{cond_a}. Let $u(t) = \Psi^t_{\rm{full}}u_0$ be the unique solution of \eqref{eq:PDE} satisfying $\sup_{[0,T]} \| u(t)\|_{H^{s+1}}\le M$, for some $M>0$.
Then there exists $C>0$ that continuously depends on $M$ and $\sup \Re(a)$ such that for any $\tau \in (0,1]$
\[
\| \Psi_{\rm full}^{\tau}(u_0) - \varphi^{\tau}_{\rm BO}e^{\tau A} u_0\|_{H^s} \le C \| a \|_{\ell^{\infty}} \tau^2.
\]
\end{lem}
Before providing the proof of Lemma \ref{eq:loc-err}, we establish the following auxiliary results. To obtain this lemma we will again use Birkhoff coordinates, however we are allowed to loose a constant $C$ and hence can obtain the result in the usual Sobolev norm, unlike for the stability bounds.

\begin{lem}[Comparison of the \eqref{BO} flow and the flow of \eqref{eq:PDE}]\label{lem:comp_BO_Eq} Let $s \geq 0$, $T>0$, and $u_0$, $v_0$ belong to $H^s_{0}$ and let $A$ satisfy \eqref{cond_a}. Assume that for all $t\in [0,T]$, $\| \Psi_{\rm full}^{t}u_0 \|_{H^s}, \| \varphi^{t}_{\rm BO} v_0 \|_{H^s} \leq M$, for some $M$. Then there exists $C>0$ that continuously depends on $M$ such that for any $t \in [0,T]$
\[
\|\Phi\left(\Psi_{\rm full}^{t}u_0\right) - \Phi\left(\varphi^{t}_{\rm BO} v_0\right)\|_{\mathfrak{h}^{s+\frac12}} \le C e^{C t} \left( \| u_0 - v_0 \|_{H^s} + T \| a \|_{\ell^{\infty}} \right).
\]
\end{lem}
\begin{proof}[Proof of Lemma \ref{lem:comp_BO_Eq}]
Let $u(t) = \Psi_{\rm full}^{t} u_0$, $v(t) = \varphi^{t}_{\rm BO} v_0$ and $t \in [0,T]$. Any solution $u$ to \eqref{eq:PDE} satisfies Duhamel's formula \eqref{eq:duh} when mapped in Birkhoff coordinates. Thus
\begin{align*}
\| \zeta_n(u(t)) - \zeta_n(v(t)) \|_{\mathfrak{h}^{s+\frac12}} & \le 
\| \left(e^{i\Omega_n(t;u)} - e^{i\Omega_n(t;v)}\right) \zeta_n(u_0)\|_{\mathfrak{h}^{s+\frac12}} + \| \zeta_n(v_0) - \zeta_n(u_0) \|_{\mathfrak{h}^{s+\frac12}} \\
&\qquad + T \sup_{\tilde t \in [0,T]} \| d_{u(\tilde t)}\zeta_n[A u(\tilde t)]\|_{\mathfrak{h}^{s+\frac12}}\\
&\le \sup_{n \ge 1}| \Omega_n(t;u) - \Omega_n(t;v)| \| \zeta_n(u_0) \|_{\mathfrak{h}^{s+\frac12}}  + C \| v_0 - u_0\|_{H^s}\\ 
&\qquad +C T \| a \|_{\ell^{\infty}}
\end{align*}
where we used the estimate \eqref{d-1} to obtain the second line. In addition, it follows from \eqref{eq:omega-n} that for  $n\ge 1$
\begin{align*}
|\Omega_n(t;u) - \Omega_n(t;v)| &\le  \int_0^T \sum_{k\ge 1}2k \left| |\zeta_k(u(\tau))|^2 - |\zeta_k(v(\tau))|^2 \right| d\tau\\
&\le C \int_0^T \| \Phi(u(\tau)) + \Phi(v(\tau))\|_{\mathfrak{h}^{ \frac12}} \| \Phi(u(\tau)) - \Phi(v(\tau))\|_{\mathfrak{h}^{ \frac12}} d\tau\\
&\le C \int_0^T \| \Phi(u(\tau)) - \Phi(v(\tau))\|_{\mathfrak{h}^{ \frac12}} d\tau
\end{align*}
where we applied $\Phi^{-1}$ and used the well-posedness of the equations to obtain the last estimate. The proof follows from a Gronwall argument.
\end{proof}
\begin{lem}[An order one bound]\label{lem:tau-bd}
Let $s\ge 0$, $\tau \in (0,1]$, $u_0 \in H^s_0$, and let $A$ satisfy \eqref{cond_a}. Assume that for all $t \in [0,\tau]$, $\| \Psi_{\rm full}^{t} u_0 \|_{H^s} \le M$, for some $M>0$. Then there exists $C>0$ that continuously depends on $s$, $M$ and $\| a \|_{\ell^{\infty}}$ (independent of $\tau$) such that
\[
\sup_{t \in[0, \tau]}\| \Psi_{\rm full}^{t} u_0 - \varphi^t_{\rm BO}(e^{\tau A} u_0) \|_{H^s} \le C \| a \|_{\ell^{\infty}} \tau.
\]
\end{lem}
\begin{proof}[Proof of Lemma \ref{lem:tau-bd}] The proof follows from Lemma \ref{lem:comp_BO_Eq} (with $T=\tau$) and the fact that for $v_0 = e^{\tau A} u_0$ we have
\begin{equation*}
\| v_0 - u_0\|_{H^s} = \| (1-e^{\tau A}) u_0 \|_{H^s} \leq c \tau \|a \|_{\ell^{\infty}} \|u_0 \|_{H^s}
\end{equation*}
for some constant $c$ that only depends on $\|a \|_{\ell^{\infty}}$.
\end{proof}
\begin{proof}[Proof of Lemma \ref{eq:loc-err}] 
Let $\tau >0$.
Define $v_0 = e^{\tau A}u_0$, and consider the flow $v(t) = \varphi_{\rm BO}^t(v_0)$ of the Benjamin--Ono equation evolving from $v_0$. It satisfies Duhamel's formula
\[
v(t) = e^{t\partial_x |D|}v_0 - \int_0^t e^{(t-r)\partial_x |D|}\partial_x v^2(r)dr.
\]
When $t= \tau$, the solution $v(\tau)$ is the Lie splitting scheme over one time step. 
Next consider the full flow $u(t)$ to \eqref{eq:PDE} written as Duhamel's formula centered about the linear term $e^{t(\partial_x |D| + A)}u_0$,
\[
u(t) = e^{t(\partial_x |D| + A)}u_0 - \int_0^t e^{(t-r)(\partial_x |D| + A)}\partial_x u^2(r)dr.
\]
We now consider the difference of $u$ and $v$ at time $\tau$: observing that the operator $\partial_x|D|$ and $A$ are diagonal in Fourier space, we have $e^{t\partial_x |D| + A} = e^{t\partial_x |D|}e^{tA}$, and hence
\[
u(\tau) - v(\tau) = -\int_0^\tau e^{(\tau-r)\partial_x |D|}\left( e^{(\tau -r)A}\partial_x u^2(r) - \partial_x v^2(r)\right) dr.
\]
Thus,
\begin{align*}
\| u(\tau) - v(\tau) \|_{H^s} &\le \tau \sup_{r\in[0,\tau]} \| e^{(\tau -r)A} \partial_x u^2(r) - \partial_x  v^2(r)\|_{H^s} \\
&\le  \tau \sup_{r\in[0,\tau]} \left( \| u^2(r) - v^2(r)\|_{H^{s+1}}+ \tau \| a \|_{\ell^{\infty}} \| \partial_x u^2(r)\|_{H^s} \right)\\
& \le C \tau \sup_{r\in[0,\tau]} \| u(r) - v(r)\|_{H^{s+1}} + C \| a \|_{\ell^{\infty}} \tau^2
\end{align*}
where $C$ is a generic constant which continuously depends on $\| a \|_\infty$ and $\| u_0\|_{H^{s+1}}$. We conclude the proof by applying Lemma \ref{lem:tau-bd} to the first term above, obtaining the desired second-order estimate.

\end{proof}

\subsection{Proof of Theorem \ref{thm:cv-time}}\label{sec:proof_thm_1}

Let $s \ge 0$, $u_0 \in H^{s+1}_0$, and $T>0$. Let $u \in C([0,T], H^{s+1}_0)$ be the unique solution to \eqref{eq:PDE}. We write the global error as a telescoping sum
\begin{equation}\label{eq:telescopic}
\left\|\left(\Psi_{\rm full}^{n\tau} - (\varphi^{\tau}_{\rm BO}e^{\tau A} )^n\right) u_0\right\|_{H^s} 
\le \sum_{j=1}^n \| \mathcal{I}_{j}\|_{H^s},
\end{equation}
where $\mathcal{I}_{j} = \left((\varphi^{\tau}_{\rm BO}e^{\tau A} )^{j-1} \Psi_{\rm full}^{\tau} - (\varphi^{\tau}_{\rm BO}e^{\tau A} )^j\right) \Psi_{\rm full}^{(n-j)\tau}u_0$. Fix $\tilde\ell \in \{1, \dots, n\}$, set $w_0 = v_0 = \Psi_{\rm full}^{(n-\tilde\ell)\tau}u_0$, and for $k \le \tilde\ell$ define 
\[
w_k := (\varphi^\tau_{\rm BO} e^{\tau A})^k w_0 \quad \text{and} \quad v_k := \Psi_{\rm full}^{k\tau} v_0.
\]
We next show that the numerical iterates $w_{k}$ remain bounded, by working in the Birkhoff coordinates.

By the well-posedness of the equation, $M_{s+1} := \sup_{t \in [0,T]} \|\Psi_{\rm full}^t u_0\|_{H^{s+1}} < \infty$. Set $M_\Phi := \sup_{t \in [0,T]} \|\Phi(\Psi_{\rm full}^t u_0)\|_{\mathfrak{h}^{s+1/2}}$ and $R_\Phi := M_\Phi + 1$. 
By Lemma \ref{t:birkhoff-bdd}, there exists $\tilde M > 0$ such that the bound $\|\Phi(f)\|_{\mathfrak{h}^{s+1/2}} \le R_\Phi$ implies that $\|f\|_{H^s} \le \tilde M$. Thus, by Lemma \ref{lem:stab-split}, there exists a uniform stability constant $\tilde{c} = c(\tilde M, \|a\|_{\ell^\infty}) > 0$ such that for any $f, g \in H^s_0$ with $\|\Phi(f)\|_{\mathfrak{h}^{s+1/2}}, \|\Phi(g)\|_{\mathfrak{h}^{s+1/2}} \le R_\Phi$:
\begin{equation}\label{eq:unif-stab-pf}
\|\Phi(\varphi^\tau_{\rm BO}e^{\tau A} f) - \Phi(\varphi^\tau_{\rm BO}e^{\tau A} g)\|_{\mathfrak{h}^{s+\frac12}} \le e^{\tilde{c}\tau} \|\Phi(f) - \Phi(g)\|_{\mathfrak{h}^{s+\frac12}}.
\end{equation}
Furthermore, by Lemma \ref{eq:loc-err} and Lemma \ref{t:birkhoff-bdd}, for any $v \in H^{s+1}_0$ bounded by $M_{s+1}$, we have
\begin{equation}\label{eq:loc-err-pf}
\|\Phi(\Psi_{\rm full}^\tau v) - \Phi(\varphi^\tau_{\rm BO}e^{\tau A} v)\|_{\mathfrak{h}^{s+\frac12}} \le C_{\rm loc} \|a\|_{\ell^\infty} \tau^2,
\end{equation}
where $C_{\rm loc}$ depends on $M_{s+1}$.

We now show that $w_k$ remain in the ball $B(0, R_\Phi) \subset \mathfrak{h}^{s+1/2}$.
We claim by induction that $\|\Phi(w_k)\|_{\mathfrak{h}^{s+1/2}} \le R_\Phi$ and $\|\Phi(w_k) - \Phi(v_k)\|_{\mathfrak{h}^{s+1/2}} \le C_{\rm loc} \|a\|_{\ell^\infty} e^{\tilde{c}T} k \tau^2$.

The base case $k=0$ is immediate. Assuming the claim holds for some $k < \tilde\ell$, we decompose the error at step $k+1$ as follows:
\begin{equation}\label{eq:decomp-pf}
\begin{split}
\Phi(w_{k+1}) - \Phi(v_{k+1}) &= \left( \Phi(\varphi^\tau_{\rm BO}e^{\tau A} w_k) - \Phi(\varphi^\tau_{\rm BO}e^{\tau A} v_k) \right) \\
&\qquad + \left( \Phi(\varphi^\tau_{\rm BO}e^{\tau A} v_k) - \Phi(\Psi_{\rm full}^\tau v_k) \right).
\end{split}
\end{equation}
By using the first bound in the induction hypothesis, $\|\Phi(w_k)\|_{\mathfrak{h}^{s+1/2}} \le R_\Phi$ (and $\|\Phi(v_k)\|_{\mathfrak{h}^{s+1/2}} \le M_\Phi < R_\Phi$), and hence by applying \eqref{eq:unif-stab-pf} and then \eqref{eq:loc-err-pf} to the first and respectively second term in \eqref{eq:decomp-pf} gives
\[
\|\Phi(w_{k+1}) - \Phi(v_{k+1})\|_{\mathfrak{h}^{s+\frac12}} \le e^{\tilde{c}\tau} \|\Phi(w_k) - \Phi(v_k)\|_{\mathfrak{h}^{s+\frac12}} + C_{\rm loc} \|a\|_{\ell^\infty} \tau^2.
\]
Unrolling this recurrence across steps $m = 0, \dots, k$ yields
\[
\|\Phi(w_{k+1}) - \Phi(v_{k+1})\|_{\mathfrak{h}^{s+\frac12}} \le C_{\rm loc} \|a\|_{\ell^\infty} \tau^2 \sum_{m=0}^k e^{m\tilde{c}\tau} \le C_{\rm loc} \|a\|_{\ell^\infty} e^{\tilde{c}T} (k+1)\tau^2,
\]
proving the error bound for step $k+1$. By the triangle inequality,
\begin{align*}
\|\Phi(w_{k+1})\|_{\mathfrak{h}^{s+\frac12}} &\le \|\Phi(v_{k+1})\|_{\mathfrak{h}^{s+\frac12}} + \|\Phi(w_{k+1}) - \Phi(v_{k+1})\|_{\mathfrak{h}^{s+\frac12}} \\
&\le M_\Phi + C_{\rm loc} \|a\|_{\ell^\infty} e^{\tilde{c}T} T \tau \le R_\Phi
\end{align*}
for all $\tau \le \tau_0$ with 
\begin{equation}\label{eq:tau_0}
\tau_0:= \min\left(1, \frac{1}{C_{\rm loc} \|a\|_{\ell^\infty} e^{\tilde{c}T} T}\right),
\end{equation}
completing the induction.

Finally, coming back to \eqref{eq:telescopic}, we apply $\Phi$ once, and apply Lemma \ref{lem:stab-split} $(j-1)$ times with the uniform constant $\tilde{c}$. Together with \eqref{eq:loc-err-pf}, this yields
\[
\|\mathcal{I}_j\|_{H^s} \le \tilde{C} e^{(j-1)\tau \tilde{c}} \|\Phi(\Psi_{\rm full}^\tau v) - \Phi(\varphi^\tau_{\rm BO}e^{\tau A} v)\|_{\mathfrak{h}^{s+\frac12}} \le C' \|a\|_{\ell^\infty} \tau^2 e^{j \tau \tilde{c}},
\]
where $v = \Psi_{\rm full}^{(n-j)\tau}u_0$. Summing over $j = 1, \dots, n$ concludes the proof:
\[
\sum_{j=1}^n \|\mathcal{I}_j\|_{H^s} \le C' \|a\|_{\ell^\infty} \tau^2 \sum_{j=1}^n e^{j \tau \tilde{c}} \le \|a\|_{\ell^\infty} C(s, T, \|u_0\|_{H^{s+1}}, \|a\|_{\ell^\infty}) e^{\tilde{c} T} \tau.
\]

\begin{rem}[The $L^2$-norm: an easier case]
The above argument is easier when taking $s=0$ as $\varphi^{\tau}_{\rm BO}$ preserves the $L^2$-norm. Indeed, given that for any $f \in L^2$
\[
\| e^{\tau A } f\|_{L^2}, \|\Psi_{\rm full}^{t} f \|_{L^2} \le e^{\tau \sup \Re(a)}\| f\|_{L^2}
\]
we have
 \[\|h_{1,\ell,\ell} \|_{L^2} = e^{\tau(\ell-1) \sup \Re(a)}\|\Psi_{\rm full}^{(n-\ell+1)\tau}u_0 \|_{L^2} \le C e^{T \sup \Re(a)} \|u_0\|_{L^2},
 \]
  and  
\[\|h_{2,\ell,\ell} \|_{L^2} \le e^{\tau\ell \sup \Re(a)} \|\Psi_{\rm full}^{(n-\ell)\tau}u_0 \|_{L^2} \le Ce^{T \sup \Re(a)} \| u_0\|_{L^2}.\]
Note that we did not need to use the stability bound of Lemma \ref{lem:stab-split} to obtain these bounds.
Thus, by applying $\Phi$ together with Lemma \ref{lem:stab-split} $j$-times we obtain the desired result. 
\end{rem}

\begin{rem}[Space discretization via the explicit formula for \eqref{BO}]\label{rem:space-approx}
To obtain a fully discrete scheme, only a spatial approximation of the flow $\varphi^\tau_{\rm BO}$ is required, as presented in Section~\ref{sec:numerics}, alongside numerical experiments. However, extending the stability analysis to the fully discrete setting in Birkhoff coordinates presents new challenges. Indeed, each Birkhoff coordinate $\zeta_n$ is expressed in terms of the eigenvalues of the underlying Lax operator and its $n$-th eigenfunction  \cite{GK21}*{Eq.~(4.1)}. Because our spatial discretization does not preserve the first $K$ eigenvalues exactly, a non-integrable perturbation error is introduced, generating an additional error term on the Birkhoff side (which can't easily be seen as a solution to a perturbed equation). Rigorously analyzing this term requires a deeper understanding of the correspondence between the schemes \eqref{eq:ABCD-scheme} for $\varphi^\tau_{\rm BO}$ and an approximation of the first $K$ Birkhoff coordinates. The fully discrete analysis thus requires dedicated analytical techniques and will be given in a future work.
\end{rem}

\subsection{Convergence in the deep-water limit for ILW}\label{sec:comparison-ILW-BO}

We recall that for the ILW equation, the perturbation symbol is given by
\[
a(k) =  i (\coth(\delta k) - \sgn(k)) k^2,
\]
which satisfies the explicit exponential decay bound for any $\delta \geq 1$:
\begin{equation}\label{estim_A_ILW}
\| a \|_{\ell^{\infty}} \leq 3e^{-2 \delta}.
\end{equation}

The following uniform a priori bound is a direct consequence of \cite{CLOP-24}*{Theorem 1.2}.

\begin{prop}\label{prop:a-priori-ILW} Let $s \geq 0$, $u_{0} \in H^s_0$. Then, for any $T>0$, there exist $\delta_0 \geq 1$ and a constant $C>0$ such that for any $\delta \ge \delta_0$,
\[
\sup_{t\in[0,T]}\| u^{\delta}(t)\|_{H^s}\le C
\]
where $u^{\delta}$ is the unique solution to \eqref{ILW} with initial datum $u_0$.
\end{prop}

Combining Lemma~\ref{lem:comp_BO_Eq} with $u_0 = v_0$ and estimate \eqref{estim_A_ILW} yields the following comparison estimate on the Birkhoff side.

\begin{prop}[Comparison between ILW and BO]\label{prop:comparison_ILW_BO}
Let $s\ge 0$, $u_0 \in H^{s}_{0}$, $M>0$, and $T>0$. Then there exist $\delta_0 \geq 1$ and a constant $C_0>0$ independent of $\delta$ (and depending only on $M, s,$ and $T$) such that for any $t \in [0,T]$ and any $\delta \geq \delta_0$,
\[
\|\Phi\left(u^{\delta}(t) \right) - \Phi\left( \varphi^t_{\rm BO} u_0 \right)\|_{\mathfrak{h}^{s+\frac12}} \le C_0 e^{C_0 T} T e^{-2 \delta},
\]
where $u^{\delta}$ is the unique solution to \eqref{ILW} and $\varphi^t_{\rm BO} u_0$ is the solution to \eqref{BO}, both evolving from the initial datum $u_0$.
\end{prop}

We are now in a position to prove our second main convergence result.
\begin{proof}[Proof of Theorem \ref{thm:cv-ILW-BO}]

Let $T>0$, $s \ge 0$, $u_0 \in H^{s+1}_0$, and $n \in \mathbb{N}$ such that $n\tau \le T$.  Let $u^{\delta}$ be the solution to \eqref{ILW} with initial datum $u_0$. By the triangle inequality, we split the total error into time-discretization error and the error between the models:
\[
\|(\varphi^{\tau}_{\rm BO} e^{\tau A})^nu_0 - \varphi^{n\tau}_{\rm BO} u_0 \|_{H^s} \leq \|(\varphi^{\tau}_{\rm BO} e^{\tau A})^nu_0 - u^\delta(n \tau) \|_{H^s} + \|u^\delta(n \tau) - \varphi^{n\tau}_{\rm BO} u_0 \|_{H^s}.
\]
For the first term in the above right-hand side, applying Theorem~\ref{thm:cv-time} together with the decay bound $\|a\|_{\ell^\infty} \le 3e^{-2\delta}$ from \eqref{estim_A_ILW} yields the desired bound.
Finally, for the second term, we note that by Proposition~\ref{prop:a-priori-ILW} and the well-posedness of \eqref{BO}, both $u^\delta(t)$ and $\varphi^t_{\rm BO} u_0$ remain in a fixed bounded subset of $H^s_0$ for all $t \in [0,T]$ and $\delta \ge \delta_0$. Therefore, the second bound follows by combining Proposition~\ref{t:birkhoff-bdd} with Proposition~\ref{prop:comparison_ILW_BO}.
\end{proof}

\begin{rem}[Size of time step $\tau$ in Theorem \ref{thm:cv-ILW-BO}]\label{rem:tau_0}
Note that Theorem \ref{thm:cv-ILW-BO} holds for all $\tau \in (0,1]$ without restricting $\tau \le \tau_0 < 1$, with $\tau_0$ defined in \eqref{eq:tau_0}. Indeed, using \eqref{estim_A_ILW}, choosing $\delta_0 \ge 1$ sufficiently large ensures that
\[
C_{\rm loc} \|a\|_{\ell^\infty} e^{\tilde{c}T} T \le C_{\rm loc} e^{-2\delta_0} e^{\tilde{c}T} T \le 1
\]
for all $\delta \ge \delta_0$, yielding $\tau_0 = 1$.
\end{rem}

\section{Numerical experiments}\label{sec:numerics}
Let $K$ denote the highest Fourier frequency used in the discretization.
Given the PDE \eqref{eq:PDE} with an associated perturbative operator $A$ and mean-zero initial data $u_0$, we introduce the following class of \textit{fully discrete schemes}:
\begin{equation}
u_K^{n+1} = \varphi^{\tau,K}_{\rm BO} e^{\tau A} u^n_K, \quad n\tau \le T, \quad u^0_{K} = \left(\Pi_{K} + \overline{\Pi}_K\right)u_0.
\label{eq:fully-discrete}
\end{equation}
with the Szeg\H o projector
$\widehat{\Pi}_{{ K}}(k) = \mathbbm{1}_{0\le k < { K}}$. 
As presented in Alama Bronsard--Laurens \cite{ABL-26}, there are multiple ways to efficiently and accurately approximate $\varphi^{\tau}_{\rm BO}$ by discretizing the explicit formula \eqref{eq:explicitForm}, yielding schemes $\varphi_{\rm BO}^{\tau,K}$ that only require an approximation in space (and which are exact for each time $\tau$). By composing this approximation with the perturbative flow $e^{\tau A}$, we obtain a robust \textit{class} of fully discrete schemes.

For the numerical simulations presented below, we implement the flow $v_{K}(t) = \varphi^{t,K}_{\rm BO}v_0$ using the simplest approximation of the explicit formula for \eqref{BO}, first introduced by Alama~Bronsard--Chen--Dolbeault \cite{ABCD-25} (see also Remark \ref{rem:l2-pres-schemes}):
\begin{align}\label{eq:ABCD-scheme}
\begin{cases}\smallskip
 \widehat{ v_{{ K}}}(t,k) =\langle (\mathrm{e}^{it}\mathrm{e}^{2it{ L}_{ v_0, {K}}}{ S^{*}})^k \Pi_{ K} v_0, 1 \rangle, \quad 0 \le k < { K}\\ \smallskip
 \widehat{ v_{ K}}(t,k) = \overline{\widehat{ v_{ K}}(t,-k)}, \quad { -K} < k<0
\end{cases}
\end{align}
with  the Lax and shift operators
\[
{L}_{v_0, { K}}f_{ K} = Df_{ K} - \Pi_{K}(v_0f_{K}), \quad { S^{*}}f_{{ K}}= \Pi (e^{-ix}f_{ K}),
\]
where 
$f_{ K} \in  L^2_{{ K}} = \Pi_{ K} L^2_{+} = \bigl\{f\in L^2: \widehat{f}(k) = 0, \ k<0 \ \text{or} { \ k \ge K}\bigr\}$. 

It is worth noting that Birkhoff coordinates serve purely as an analytical tool for our proof of convergence. In practice, the scheme is implemented directly in Fourier space without transforming into Birkhoff coordinates.

\begin{rem}\label{rem:l2-pres-schemes}
While we utilize the scheme of \cite{ABCD-25} for its simplicity, this approach has been generalized by Alama Bronsard--Laurens \cite{ABL-26} to include more subtle approximations of the explicit formula that guarantee exact discrete $L^2$-norm preservation. When the perturbative term $A$ generates an isometry (i.e., $\|e^{t A} v\|_{L^2} = \|v\|_{L^2}$, $t>0$), as is for example the case for \eqref{ILW}, \eqref{Smith} and \eqref{KdVBO}, utilizing these refined approximations (see \cite{ABL-26}*{Example 2.7}) ensures that the resulting fully discrete scheme \eqref{eq:fully-discrete} preserves the discrete mass over time:
$$ \|u_K^{n+1}\|_{L^2} = \|\varphi^{\tau,K}_{\rm BO} e^{\tau A} u^n_K\|_{L^2} = \|e^{\tau A} u^n_K\|_{L^2} = \|u^n_K\|_{L^2}, \quad n\tau \le T, $$
where the second equality follows from \cite{ABL-26}*{Proposition 2.8}, and the third relies on the isometry of $e^{\tau A}$ on $L^2$.
\end{rem}

\begin{rem}[Convergence of the schemes based on the explicit \eqref{BO} formula]\label{rem:ABCD-err}
In the work \cite{ABCD-25} the authors construct the scheme \eqref{eq:fully-discrete} and prove quantitative decay rates for smooth enough solutions: for any $t\in \mathbb{R}$, $0\le r\le s$ and $u_0 \in H^s(\mathbb{T})$ with $s>1$, there exists $C>0$ such that
\[
\|u_K(t) - u(t)\|_{H^r} \le C(1+t)K^{-s+1+r}.
\]

In contrast the work \cite{ABL-26}  introduce a class of schemes, generalizing that of \cite{ABCD-25}, and employs a different set of techniques tailored to treat data in $L^2$, where no additional regularity is demanded on the solution i.e. convergence in $L^2$-norm for  data in $L^2$.
\end{rem}

\subsection{Error plots and energy preservation for ILW}

Throughout this subsection, we present simulations for the \eqref{ILW} equation using randomized initial data in $H^1_0(\mathbb{T})$, the construction of which we detail next.
\newline

{\bf The initial data.} We construct the Fourier coefficients $(\widehat{U}_{k})_{|k| < K}$ by sampling from Gaussian distributions. Specifically, for $0 < k < K$, we draw independent complex Gaussians $\widehat{U}_k \sim \mathcal{N}(0,1/2) + i\mathcal{N}(0,1/2)$. We then enforce the symmetry condition $\widehat{U}_{-k} = \overline{\widehat{U}_{k}}$ for $-K < k < 0$, and lastly take $\widehat{U}_0 =0$.
The initial datum is then defined as
\begin{equation}
\label{eq:randomID}
u_{0}(x) = \sum_{-K < k < K} (1+|k|)^{-(1+\theta)}\widehat{U}_{k} e^{ikx}, \quad x\in \mathbb{T}.
\end{equation}
We normalize the function in the $L^2$-norm by setting $u_{0} \leftarrow u_{0}/\| u_{0} \|_{L^2}$. We require $\theta > 1/2$ to guarantee that the limiting function as $K \to \infty$ belongs to $H^1_0(\mathbb{T})$.
\newline

In Figure \ref{fig:temporal_convergence} we plot the temporal convergence in $L^2$-norm for the initial data \eqref{eq:randomID}, we see the expected first order convergence proven in Theorem \ref{thm:cv-time}.

\begin{figure}[htbp]
    \centering
    \includegraphics[width=0.7\textwidth]{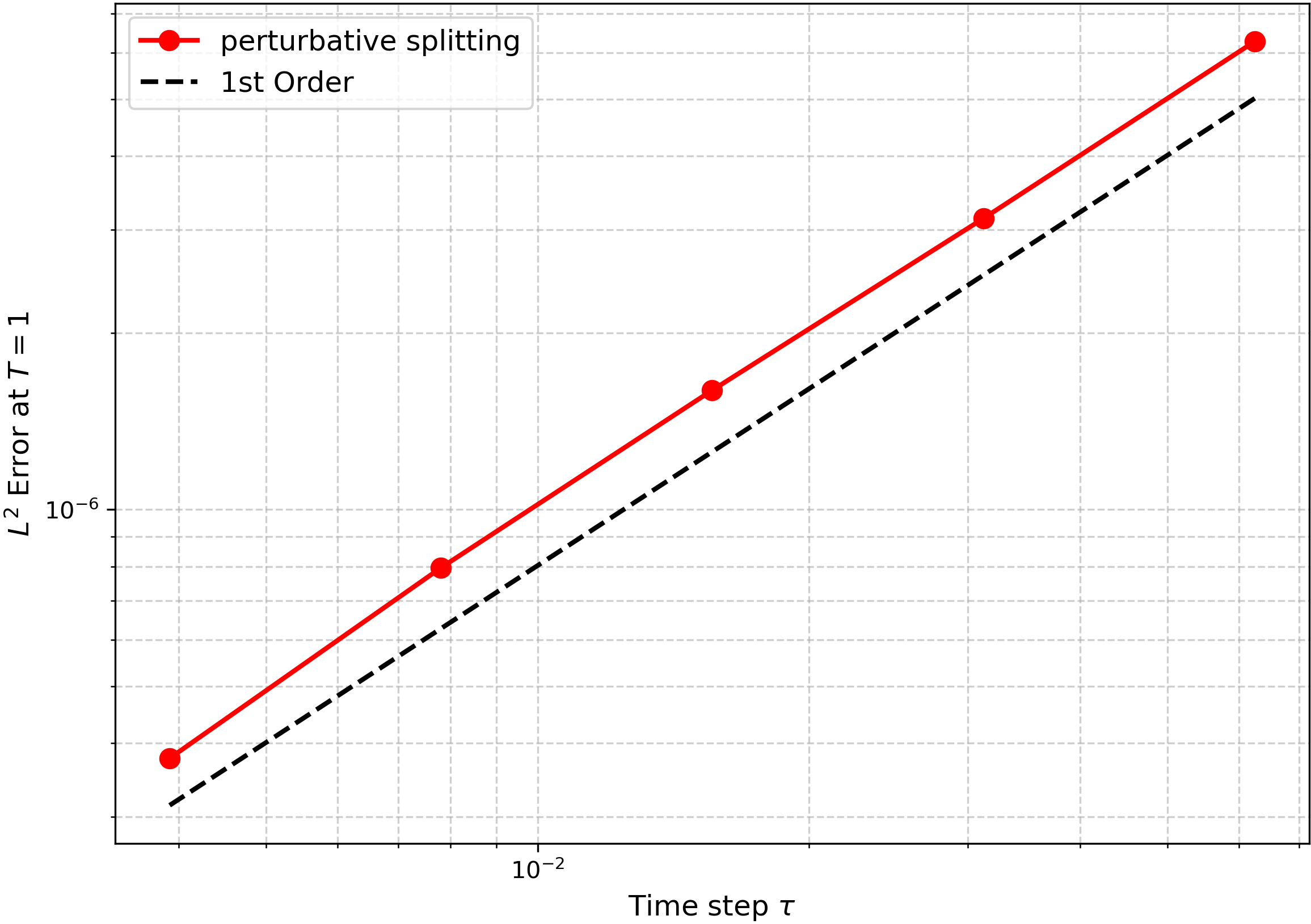}
    \caption{Temporal $L^2$ error at time $T=1.0$ for the perturbative splitting scheme \eqref{eq:fully-discrete} applied to randomized $H^1_0$ initial data. The spatial grid is fixed at $K=128$. The time step $\tau$ varies from $1/256$ to $1/16$, and the error is calculated relative to a reference solution computed with $\tau=1/2048$. The dashed black line confirms the theoretical first-order convergence in time.}
    \label{fig:temporal_convergence}
\end{figure}

\subsubsection{Comparison with classical approaches}
To evaluate the performance of our proposed method, we compare it against two standard schemes: the explicit fourth-order Runge-Kutta (RK4) method and the  Lie splitting method. We recall that the Lie splitting method is given in \eqref{eq:lie-splitting} (consisting of a splitting between the nonlinear Burgers flow and the linear perturbative propagator), which we couple with a pseudo-spectral method in space using a standard explicit integration step. While, to the best of our knowledge, no convergence results exist for these schemes applied to the ILW equation, they are natural candidates for its efficient approximation, and convergence analyses for splitting methods do exist for other related
dispersive equations sharing Burgers nonlinearity, notably, in \cites{HKRT-11, HLR-13, rousset-22} for the KdV equation and \cite{DHKR-15} for the BO equation.
\newline

We consider the preservation of the energy over time. For the ILW equation, the conserved energy is given by
\begin{equation}
    \mathcal{H}(u) = - \int \left( \frac{1}{2} u \mathcal{M} u + \frac{1}{3} u^3 \right) dx,
\end{equation}
where the operator $\mathcal{M}$ is defined via the Fourier multiplier $\widehat{\mathcal{M}u}(k) = k \coth(\delta k) \hat{u}(k)$.

\begin{figure}[h]
    \centering
    \includegraphics[width=0.95\textwidth]{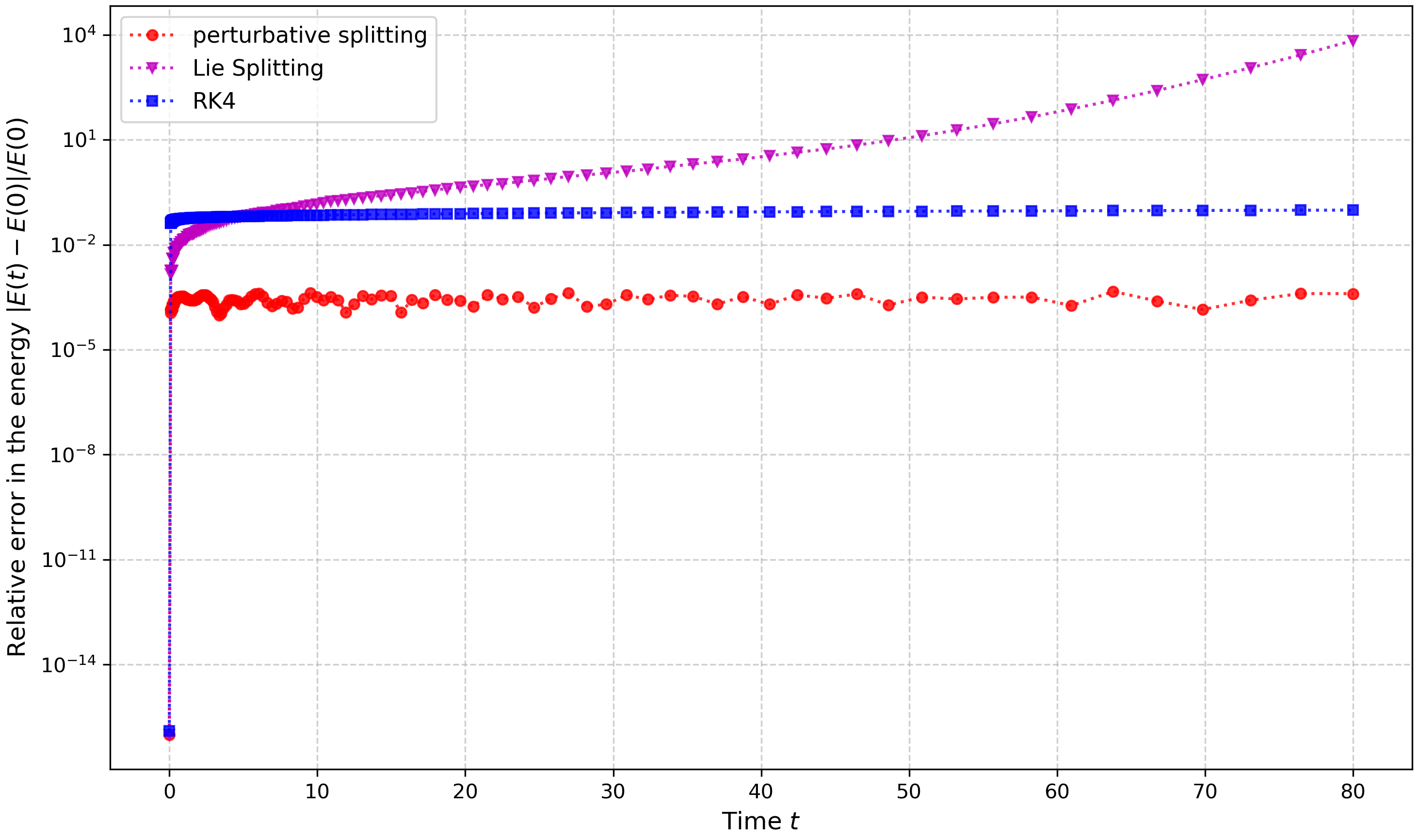}
    \vspace{0.3cm}
    \includegraphics[width=0.95\textwidth]{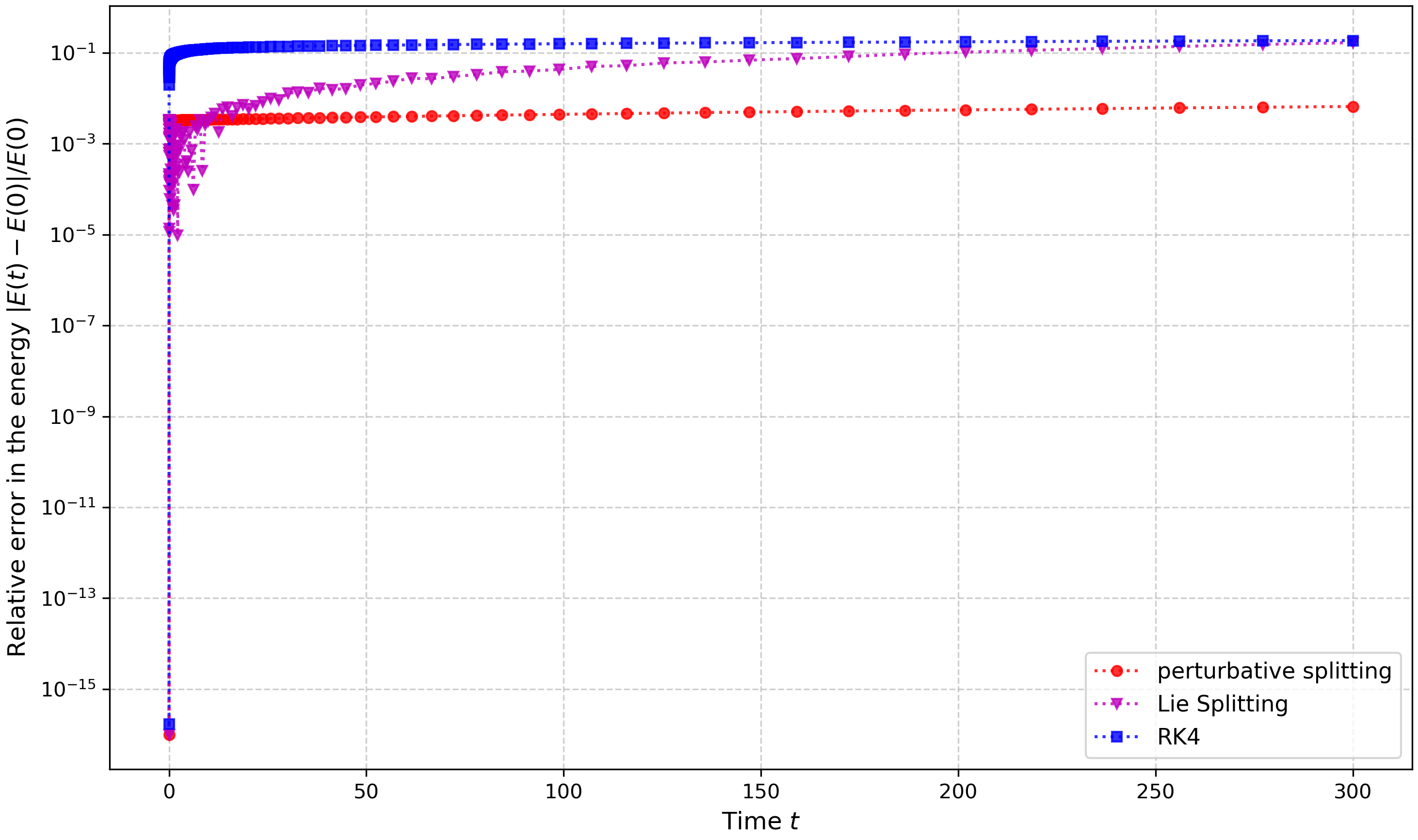}
    \caption{Relative error of the energy associated to \eqref{ILW}  for randomized $H^1_0$ initial data at a fixed spatial resolution of $K=64$. Top panel ($T=80$): computed using a quadratic time step ($\tau \sim K^{-2}$) for RK4, and a {\it linear} time step ($\tau \sim K^{-1}$) for the classical Lie (purple) and perturbative splitting schemes (red). Bottom panel ($T=200$): computed using a quadratic time step ($\tau \sim K^{-2}$) for all schemes. The Lie splitting method nearly preserves the energy only under the quadratic time step condition, in contrast to our new scheme.}
    \label{fig:energy_short_long}
\end{figure}

\begin{figure}[h]
    \centering
    \includegraphics[width=0.95\textwidth]{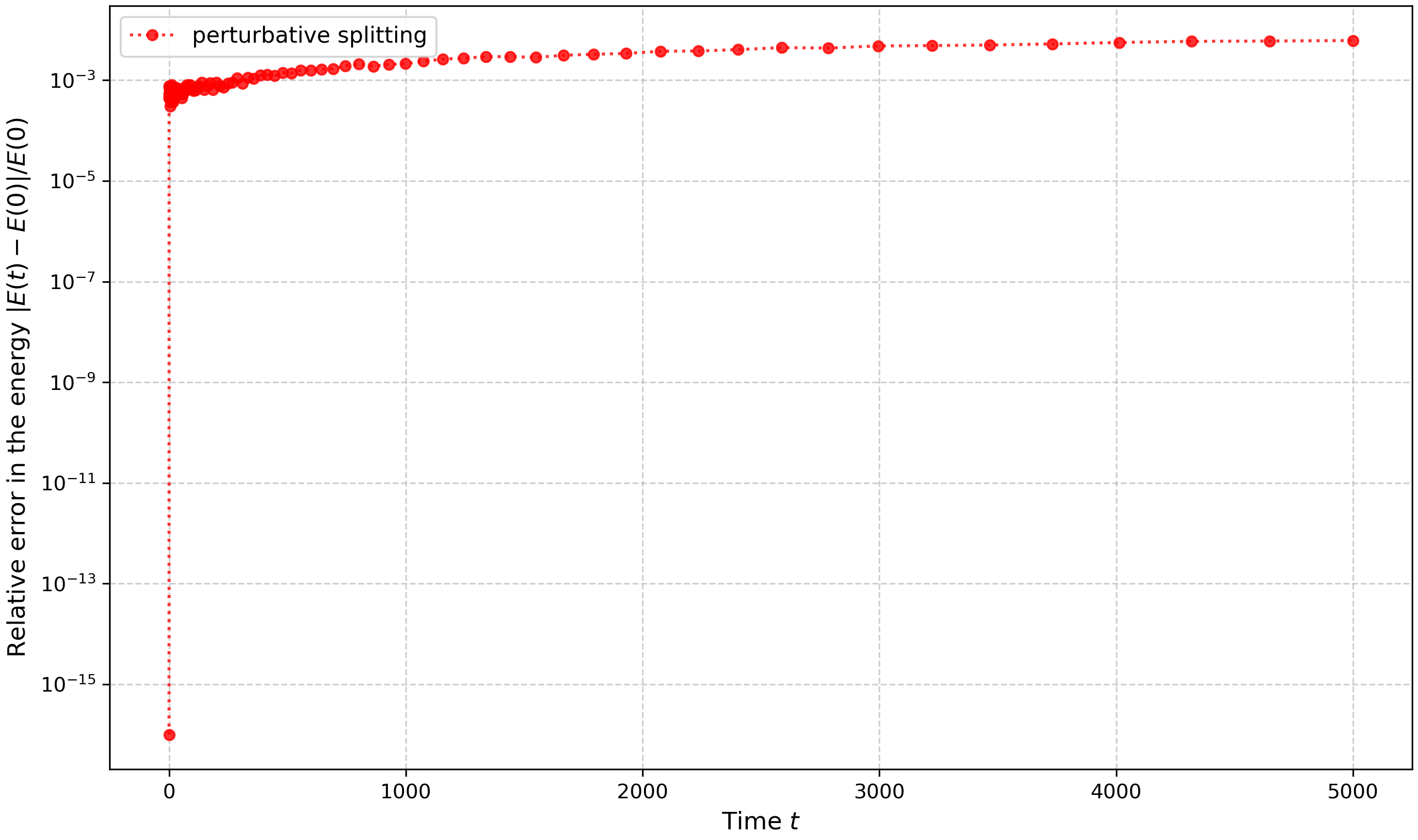}
    \caption{Long-time near-preservation for the energy associated to \eqref{ILW} of the perturbative splitting scheme up to $T=5000$, evaluated under a {\it linear} time step constraint ($\tau \sim K^{-1}$) at a fixed spatial resolution of $K=64$.}
    \label{fig:energy_very_long}
\end{figure}

Figures~\ref{fig:energy_short_long} and \ref{fig:energy_very_long} present the numerical error in the energy as a function of time. As shown in the top panel of Figure~\ref{fig:energy_short_long}, computing the classical Lie splitting scheme under a linear time step constraint ($\tau \sim K^{-1}$) results in an $O(1)$ error by $T=80$. Enforcing a quadratic time step constraint ($\tau \sim K^{-2}$) ensures near-preservation of energy for the Lie splitting scheme (bottom panel), as well as for the RK4 scheme which in addition requires this condition for numerical stability. For rigorous works in this direction we refer to Remark \ref{rem:energy-split}.
In contrast, the perturbative splitting scheme remains stable and {\it nearly preserves the energy  without this restrictive time step condition}; namely it maintains a lower error magnitude while operating under the linear time step constraint ($\tau \sim K^{-1}$), see also Remark \ref{rem:symm-low-reg} for a discussion on another class of schemes for which this was observed in simulations.

Figure~\ref{fig:energy_very_long} illustrates the long-time behavior of the perturbative splitting scheme up to $T=5000$. Under the linear step size constraint ($\tau \sim K^{-1}$), without a quadratic time step condition, the error in the energy remains bounded over long times. These results demonstrate that the new scheme achieves near-preservation of energy without requiring the restrictive quadratic time step condition mandatory for the classical RK4 or splitting methods.

\begin{figure}[htbp]
    \centering
    \includegraphics[width=0.8\textwidth]{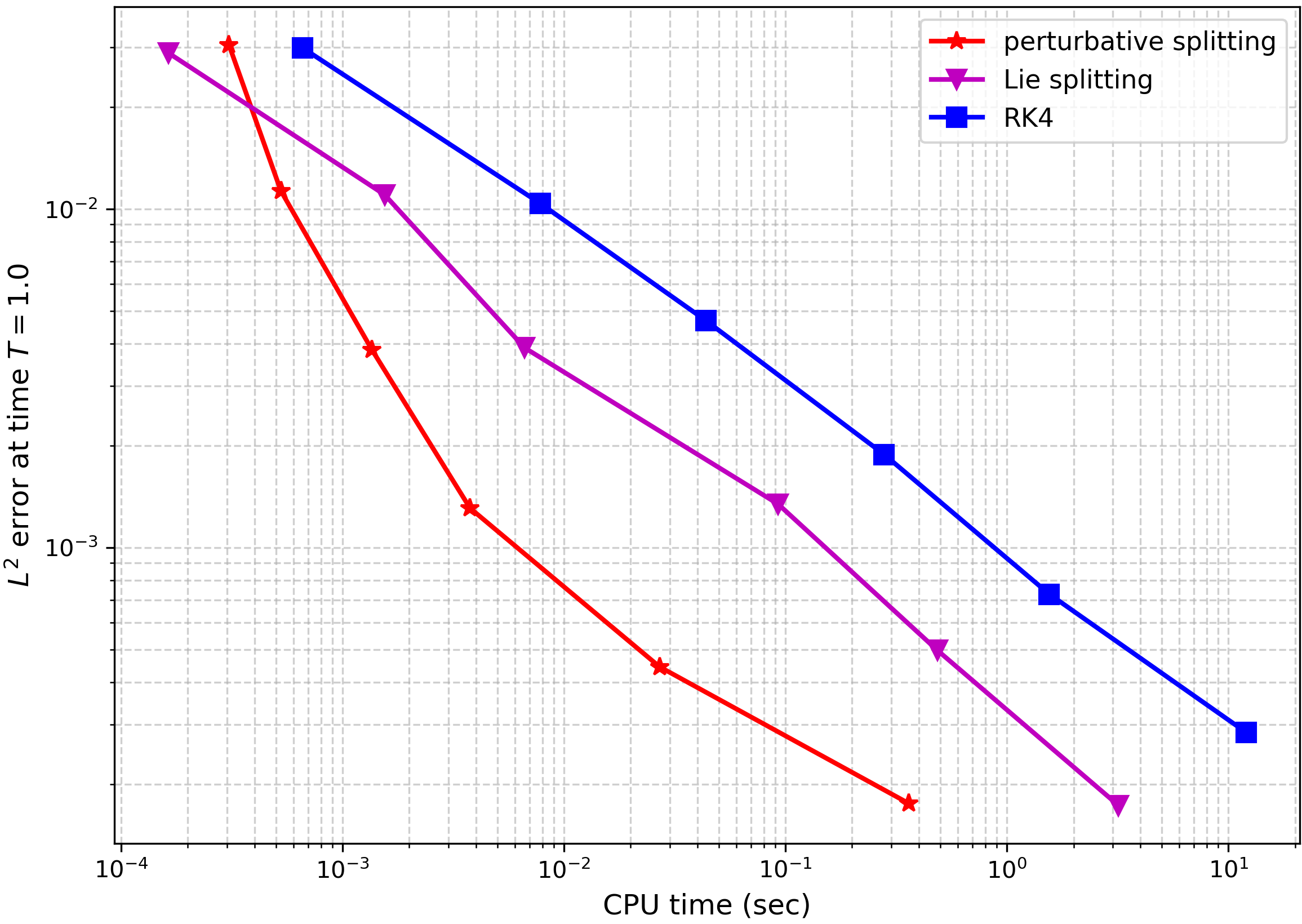}
    \caption{Convergence plot for the ILW equation in $L^2$ against computational cost at time $T=1.0$ for randomized $H^1_0$ initial data. We choose the number of Fourier modes $K$ to be powers of two ranging from $8$ to $256$. }
    \label{fig:cpu_vs_error}
\end{figure}

 In Figure \ref{fig:cpu_vs_error}, we plot the $L^2$ error against CPU cost. To ensure long-time near-energy preservation (and stability for RK4), the usual Lie splitting (purple) and RK4 schemes (blue) are constrained by a quadratic time step condition ($\tau \sim K^{-2}$). In contrast, the new perturbative splitting (red) remains stable and nearly preserves energy without a quadratic time step restriction. Operating under a linear time step ($\tau \sim K^{-1}$), it yields a highly competitive scheme.

\begin{rem}[Near-preservation of energy for classical splitting methods]\label{rem:energy-split}
In the context of splitting methods for the nonlinear Schr\"odinger equation, the necessity of a quadratic time step condition to nearly preserve energy is rigorously established by Faou~\cite{Faou2012}. There, the numerical scheme is shown to (almost) exactly solve a modified PDE at each time step. For smooth solutions and under a quadratic time step constraint, the resulting modified energy can be shown to remain close to the exact energy. We also refer to \cites{Faou2010, Faou2010II, Faou2011} and \cites{Cohen2008, Gauckler2010}, which employ normal form techniques and modulated Fourier expansions, respectively, to obtain near-preservation of the energy over exponentially long times.

The rigorous justification of the near-preservation of energy for our scheme \eqref{eq:fully-discrete}, leveraging the integrable structure of the \eqref{BO} equation, is the subject of ongoing work.
\end{rem}

\begin{rem}[Symmetric low-regularity schemes and energy preservation]\label{rem:symm-low-reg}
Symmetric low-regularity time integrators have recently been introduced 
for solving local semi-linear equations, see \cites{AlamaBronsard2024, Feng2025, AlamaBronsard2026} and references therein. While more intricate to construct than classical integrators, these schemes have the advantage that a quadratic time step restriction is not needed in the numerical simulations to obtain near-energy preservation over long times. Rigorous analysis for these methods remains open; establishing such results will require novel techniques, as existing frameworks inherently yield a quadratic time step condition (see Remark \ref{rem:energy-split}).
 \end{rem}

\subsection{Long-time dynamics on $\mathbb{R}$ : soliton resolution conjecture}

For a general introduction on theoretical and numerical works regarding the soliton resolution conjecture for nonlinear dispersive PDEs we refer to the book of Klein--Saut \cite{KS-21}*{Chapter 3.5.4}. We add the fact that using Gérard's explicit formula for \eqref{BO} on~$\mathbb{R}$ \cite{G-23}, the soliton resolution conjecture has recently been solved by Gassot--Gérard--Miller~\cite{Gassot2026}. For a first theoretical study of the long-time dynamic for \eqref{ILW}, we refer to Ifrim--Saut \cite{IS-25}, who considered small initial data and applied modified energy methods adapted for quasilinear equations.

In what follows, we perform long-time numerical simulations for two equations in the class \eqref{eq:PDE}: the \eqref{ILW} and the \eqref{KdVBO} equations, starting from various generic initial data. 

In this section, our setting is the full line $\mathbb{R}$. Hence, to perform numerical simulations, we rescale the periodic problem \eqref{eq:PDE} to a large torus and ensure that the solutions remain localized within the computational domain. In the context of the \eqref{BO} equation, the rigorous justification for this large-torus approximation is established in \cite{ABM-26}, which proves the convergence of the explicit formula on the torus to that on the real line, both on the continuous and discrete level.
\newline

\noindent
{\bf The ILW equation.}
To explore the long-time behavior of the ILW equation, Figure~\ref{fig:sol-res-ilw} tracks the evolution of initial data corresponding to a positive and a negative gaussian with a depth parameter of $\delta = 1$. The solution is shown to asymptotically decompose into one distinct solitary wave as well as small-amplitude dispersive radiation.
\newline

\begin{figure}[htbp]
\centering
\includegraphics[width=0.8\textwidth]{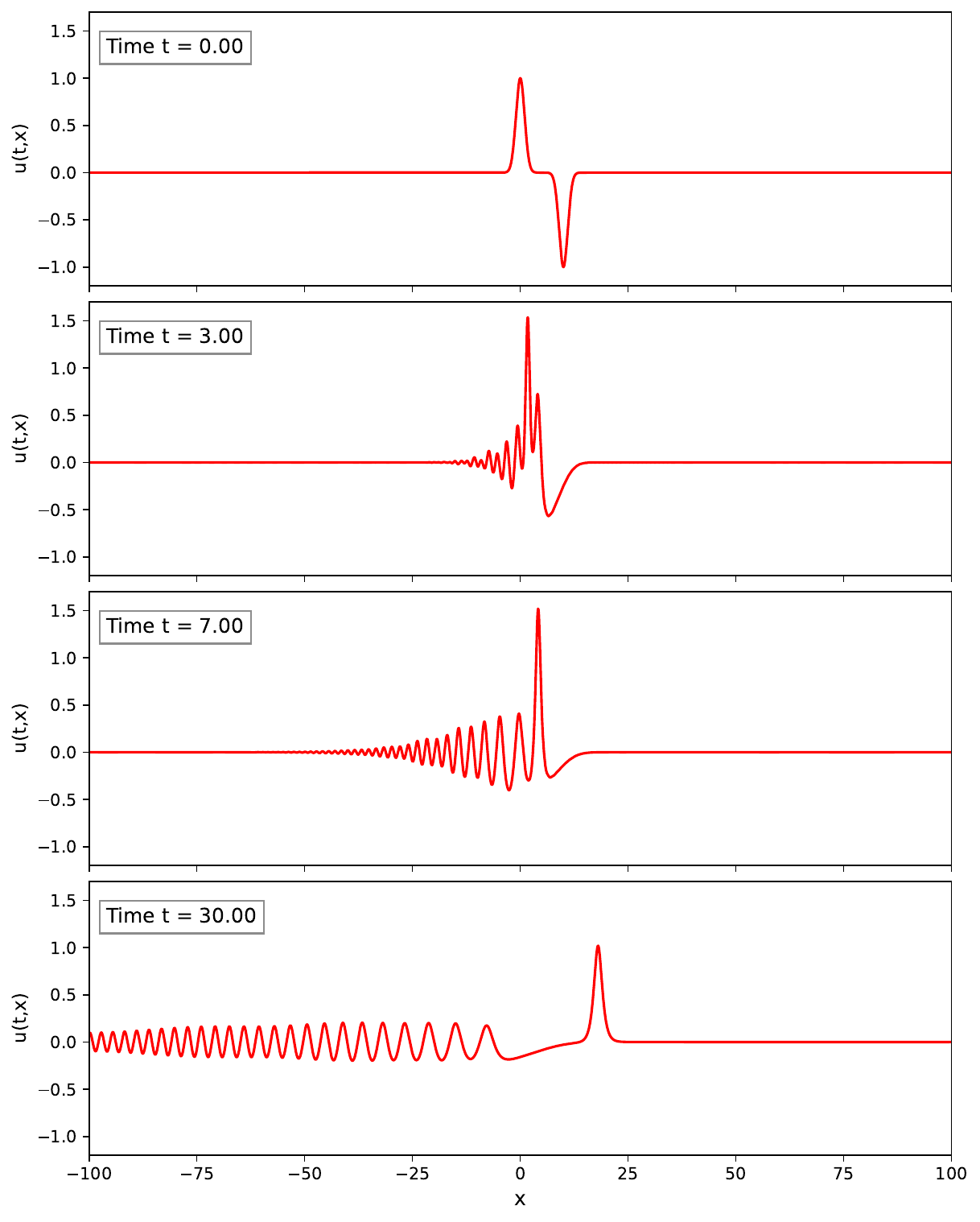} 
    
\caption{Soliton resolution for the ILW equation on the spatial domain $[-80\pi, 80\pi]$, starting from two gaussians, one positive and one negative. We took $\delta = 1$, $K = 2^{10}$, and a linear time step condition.}
\label{fig:sol-res-ilw}
\end{figure}

\noindent
{\bf The KdV--BO equation.}
While the \eqref{ILW} is a well known equation, the  \eqref{KdVBO} equation is new.
We observe long-time dynamics consistent with the soliton resolution conjecture for this equation as well. 

In Figure \ref{fig:sol-res-kdvbo} we show the interaction of generic initial data, specifically the case of two positive gaussians.  

To further investigate this phenomenon, we transition from generic Gaussian profiles to consider the \eqref{KdVBO} equation subject to the rational initial data 
\begin{equation}
    u_0(x) = \frac{2c}{1+x^2}
\end{equation}
where $c>0$ represents the mass multiplier.

For the \eqref{BO} equation, taking $c = 1$ in the above corresponds to the exact traveling wave. Furthermore, when $c = N$ (for integer $N$), this initial profile generates an exact $N$-soliton solution with no dispersive radiation \cite{Gassot2026}. For intermediate values $c \in (N-1, N)$, the BO equation resolves into exactly $N$ right-propagating solitons accompanied by left-propagating dispersive radiation.

Our numerical experiments reveal a different behavior for the KdV--BO equation. Specifically, in addition to the radiating background, which we observe for each of these values of $c$, we observe $N$ asymptotic solitons when $c \in \left(\frac{3}{2}(N-1), \frac{3}{2}N\right]$.
Thus we believe that the number of solitons $N$ can be expressed directly as a function of the mass multiplier $c$ using the ceiling function
$N = \left\lceil \frac{2c}{3} \right\rceil$.

While this is supported by numerical evidence, a rigorous theoretical justification remains open.

\begin{figure}[htbp]
\centering
    \includegraphics[width=0.8\textwidth]{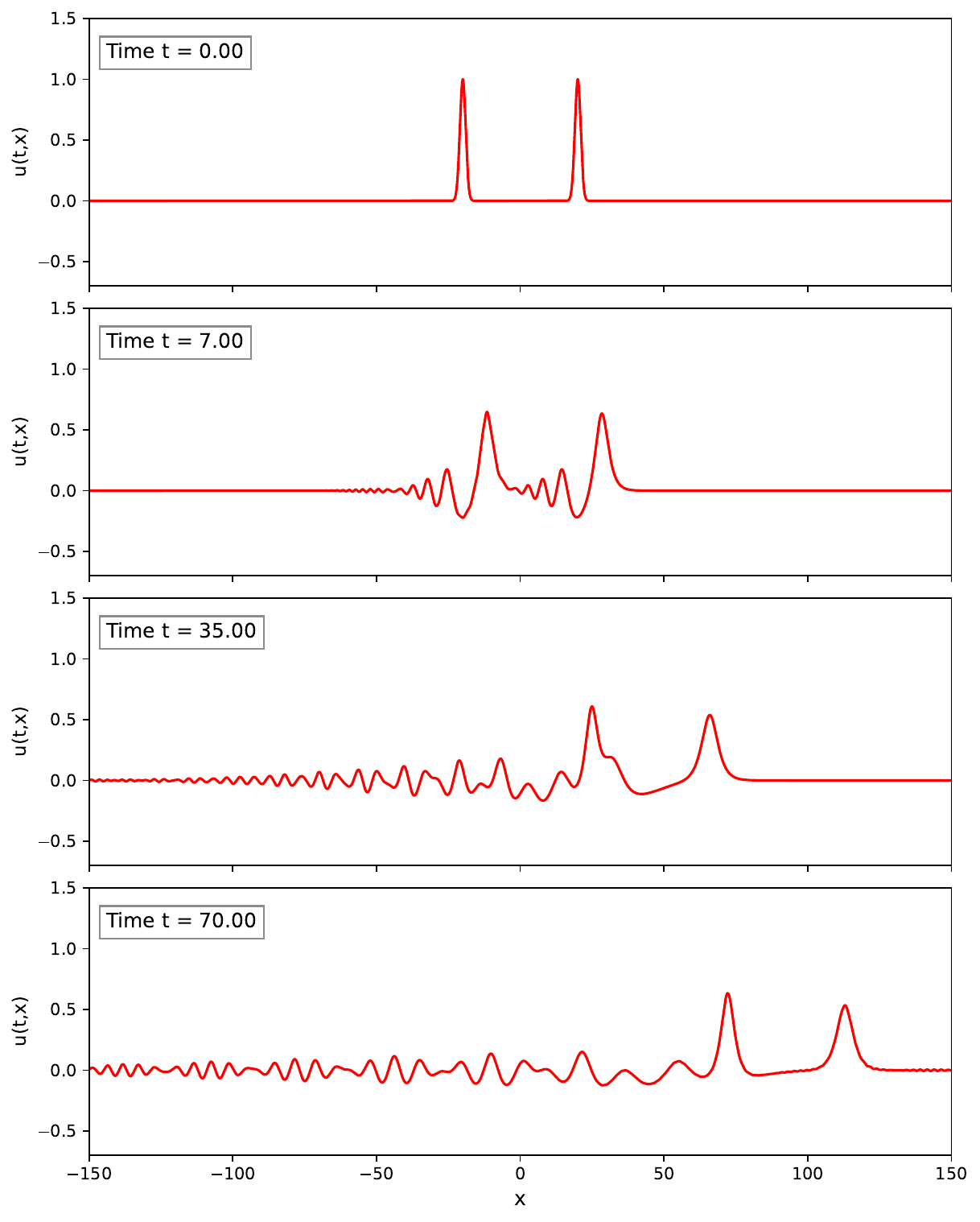} 
\caption{Soliton resolution for \eqref{KdVBO} starting from two positive Gaussians on the spatial domain $[-80\pi, 80\pi]$, computed using $K = 2^{10}$ modes under a linear time step condition.}
\label{fig:sol-res-kdvbo}
\end{figure}

\section*{Acknowledgements}
\noindent
The work of Y.A.B. was funded by the National Science Foundation through the award
DMS-2401858.
The research of B.M was partially supported by the ANR Project HEAD ANR-24-CE40-3260. The authors thank Thierry Laurens for helpful discussions on a priori estimates and Birkhoff coordinates.

\end{document}